\documentclass[12pt]{article}
\usepackage[utf8]{inputenc}
\usepackage{amsmath}
\usepackage{amssymb}
\usepackage{amsthm}
\usepackage{bm}
\usepackage{euler,eucal}
\usepackage{textcomp}
\usepackage[T1]{fontenc}
\usepackage[english]{babel}
\usepackage{caption}

\usepackage{algorithmic}
\usepackage[ruled]{algorithm}

\newtheorem{definition}{Definition}
\newtheorem{lemma}{Lemma}
\newtheorem{remark}{Remark}
\newtheorem{corollary}{Corollary}

\def\diag{\mathrm{diag}}

\def\alert[#1]{{\large{[\bf{#1}]}}}
\let\eps=\varepsilon

\usepackage[unicode=true,plainpages=false]{hyperref}
\hypersetup{colorlinks=false,pdfborder={0 0 0.1},linkcolor=magenta,anchorcolor=magenta,urlcolor=blue,citecolor=blue}

\usepackage{tikz}
\usetikzlibrary{calc}
\usetikzlibrary{shapes}
\usetikzlibrary{fit}

\usepackage[a4paper]{geometry}
\begin{document}
\bibliographystyle{siam}

\title{Simultaneous state-time approximation of the chemical master equation using tensor product formats
\thanks{
Partially supported by RFBR grants 12-01-00546-a, 11-01-12137-ofi-m-2011, 11-01-00549-a, 12-01-33013 mol-ved-a, 12-01-31056,
by Rus. Gov. Contracts $\Pi$1112, 14.740.11.0345, 14.740.11.1067, 16.740.12.0727, and Priority Research Program OMN-3
at Institute of Numerical Mathematics, Russian Academy of Sciences, Russia, 119333 Moscow, Gubkina 8.
         }
}
\author{Sergey Dolgov, Boris Khoromskij \\
\small{Max-Planck Institute for Mathematics in Sciences,} \\
\small{Germany, 04103 Leipzig, Inselstra{\ss}e 22} \\
\small{E-mail: {[sergey.dolgov,bokh]@mis.mpg.de}}}

\date{12 November 2013}

\maketitle

\begin{abstract}

{\emph{This is an essentially improved version of the preprint \cite{dkh-cme-2012}.
This manuscript contains all the same numerical experiments, but some inaccuracies in the description of the modeling equations are corrected.
Besides, more detailed introduction to the tensor methods is presented.
}}

\vspace*{0.5cm}

We study the application of the novel tensor formats (TT, QTT, QTT-Tucker) to the solution of $d$-dimensional chemical master equations, applied mostly to gene regulating networks (signaling cascades, toggle switches, phage-$\lambda$).
For some important cases, e.g. signaling cascade models, we prove good separability properties of the system operator.
The Quantized tensor representations (QTT, QTT-Tucker) are employed in both state space and time,
and the global state-time $(d+1)$-dimensional system is solved in the structured form by using the ALS-type iteration.
This approach leads to the logarithmic dependence of the computational complexity on the system size.
When possible, we compare our approach with the direct CME solution and some previously known approximate schemes, and observe a good potential of the newer tensor methods in simulation of relevant biological systems.

\end{abstract}



\section{Introduction}
The paper is devoted to the solution of the chemical master equation in structured tensor formats.
This problem arises mostly in the modeling of chemical reactions (kinetics) in genetic regulatory networks, cell systems and so on.
Typically the number of molecules of given chemical species in intra-cellular systems is of the order of hundreds.
At such concentrations, the simulation in terms of ordinary differential equations is inappropriate, since stochastic fluctuations in numbers of molecules (of the relative order $10^{-1}$) play
an important role in the evolution of the system \cite{steuer-cell-stoch-2004,arkin-stoch-kinetic-1998}.
In such a case, we cannot ignore the stochasticity of the system.
The stochastic description of the chemical reaction kinetics is given by the chemical master equation (CME), see \eqref{cme:eqn:cme} below, which simulates the probability distribution describing the state of the system \cite{vankampen-stochastic-1981,gillespie-ssa-1976,Gillespie-rigorous-CME-1992}.

During the history of the chemical kinetics modeling in biology, different approaches have been developed.
Monte Carlo methods are based on a statistically large
ensemble of realizations of the stochastic process associated with the CME.
The most famous is the stochastic simulation algorithm (SSA) by Gillespie \cite{gillespie-ssa-1976}.
Several improvements include the advanced sampling techniques \cite{hemberg-perfect-sampling-2007},
$\tau$-leaping methods \cite{gillespie-tau-leap-2001},
system-partitioning hybrid methods \cite{goutsias-quasiequilibrium-cme-2005,hellander-hybridcme-2007,jahnke-cme-piecewise-2012}.
Additionally, the chemical Fokker-Planck equation may be considered \cite{gillespie-cme-fpe-2002} to treat high-concentration systems, which can be discretized on coarser grids than the CME.

%


In order to analyze a stochastic reaction system, all Monte Carlo techniques require a lot of realizations.
Sometimes important biological events may be very rare, and extremely large number of realizations may be required to catch the relevant statistics.

A principal alternative to the Monte Carlo-type methods is the solution of the master equation directly as a linear ODE.
For many systems, the probability distribution vanishes rapidly outside a bounded domain.
Thus, it is possible to truncate the state space to a finite domain, and approximate the exact solution by the resulting truncated one \cite{munsky-fsp-2006}.

However, even the truncated state space volume is usually still very large and grows exponentially with the number of species.
This problem is called the \emph{curse of dimensionality} since \cite{bellman-dyn-program-1957}.
For example, a system with $10$ species and typical concentrations of about $100$ would be described by $100^{10}$ float numbers, which is infeasible with any supercomputer.
So, some low-parametric approximation is needed.

As such reduced grid-based methods, we would like to mention the sparse grids technique \cite{hegland-cme-2007},
as well as the initial tensor-structured approaches: the greedy algorithm in the canonical tensor format \cite{Ammar-cme-2011,Sueli-greedy_FP-2011,maday-greedy-2009,Cances-greedy-2011,Binev-conv_rate_greedy-2011,hegland-cme-2011}, and the solver based on the Tucker manifold dynamics via the projection onto the tangent space \cite{jahnke-cme-2008}.
The latter technique exploits the so-called \emph{Dirac-Frenkel} principle to propagate the system on the tensor product manifold.
For a detailed description we refer to \cite{lubich-koch-dynten-2010,lubich-dynht-2012,okhs-dyntt-2012}.

Alternatively, one can formulate an implicit time propagation scheme (Euler, Crank-Nicolson) and apply some tensor-structured solver to the linear system involved.
Moreover, we may consider the time as an independent variable and solve the coupled state-time system at once, taking further benefits from the reduction of complexity based on the tensor structuring.
This technique can be considered as an adaptive construction of the tensor manifold without a priori knowledge on the system, in contrast to the propagation of a predefined manifold in the tensor Dirac-Frenkel scheme.
The initial application of this approach to the Fokker-Planck equation was given in \cite{DKhOs-parabolic1-2012}, and here we shall use it for the chemical master equation as well.
Our results have been announced in \cite{dkh-cme-2012}, similar scheme was considered in \cite{kkns-cme-2013}.
In addition, a theoretical study of the tensor properties for some class of steady solutions was presented in \cite{kkns-cme-theory-2013}.

A certain part of this paper contains theoretical analysis of the tensor representation of the linear system matrix, which demonstrates a nice structure.
As a by-product, we prove TT rank estimates for the Hamiltonians of the Heisenberg (XYZ) spin system models.

A good separability of the initial data (operator and initial state) is a first ingredient of the tensor structured solution scheme.
However, low ranks of the inputs do not guarantee low ranks of the solution.
In some particular cases we may derive some estimates from the smoothness, but in a more general problem, we may only assume that a low-parametric representation of solution is possible.
Another principal component of the solution scheme is a linear solver which yields a low-rank solution if it exists in principle.
In our work, we use the Alternating Minimal Energy (AMEn) method \cite{ds-amr1-2013,ds-amr2-2013}.
It employs the ideas of the alternating least squares (ALS) \cite{kolda-review-2009,white-dmrg-1993,holtz-ALS-DMRG-2012,DoOs-dmrg-solve-2011}, improved by combining with the Steepest Descent iteration to determine a new correction direction.
As a result, we obtain the globally convergent iterative method for the linear system solution in a tensor representation.
The initial verification of the AMEn methods and their comparison with the ALS
has involved some examples of the Fokker-Planck and master equations \cite{ds-amr2-2013}.
It was found that the simpler ALS technique fails to solve nonsymmetric systems (which appear in the CME simulation), and reformulation of the problem in the normal equations is too expensive, whereas the AMEn method is much faster and provides a desired accuracy, see Section \ref{sec:amrs} for an additional discussion.
Here we focus on a more thorough study of CME models.



The paper is organized as follows.
In Section 2, the chemical kinetic and master equations are presented, and basic properties are mentioned.
Section 3 describes tensor formats, principal techniques and corresponding notations.
Section 4 is devoted to some typical CME operators and their analytic tensor-structured representations.
In Section 5 we present the main computational schemes for the time propagation, as well as the stationary solution.
Finally, Sections 6 and 7 provide the illustrative numerical experiments and conclusion.

\section{Problem statement}\label{sec:cme_formul}
Suppose that $d$ different active chemical species $S_1,...,S_d$ in a well-stirred medium can react in $M$ reaction channels.
Denote the vector of their concentrations $\mathbf{x} = (x_1,...,x_d) \in \mathbb{R}_+^d$, $\mathbb{R}_+ = \{x \in \mathbb{R}: x \ge 0\}$.
Each channel is specified by a \emph{stoichiometric vector} $\mathbf{z^m} \in \mathbb{Z}^d$, and a \emph{propensity} function $w^m(\mathbf{x}):~\mathbb{R}_+^d \rightarrow \mathbb{R}_+, ~m=1,...,M$.

The deterministic ODE on the concentrations is written as follows,
\begin{equation}
\dfrac{dx_i}{dt} = \sum\limits_{m=1}^M z^m_i w^m(\mathbf{x}), \quad i=1,...,d, \quad x_i \ge 0.
\label{cme:eqn:ode_classic}
\end{equation}
However, this model describes the process good enough only in the ``classical'' limit, when the amount of molecules in the system is very large (of the order of the Avogadro constant).
In some biological systems (cells, viruses), the species are contained in very small amounts w.r.t. the system volume, e.g. tens-thousands molecules.
In this case, the reaction occurrence is a stochastic process, and the quality of the model \eqref{cme:eqn:ode_classic} becomes very poor.

Moreover, if we measure concentrations $x_i$ in terms of the numbers of molecules (\emph{copy numbers}), it does not hold $x_i \in \mathbb{R}$, but $x_i \in (\{0\}\cup\mathbb{N})$ instead, since molecules of each specie are assumed to be non-divisible (there may be of course splitting reactions, but their products should be considered as separate species).
In the following, by ``concentration'' $x_i$, we always mean the copy number of a specie, so that $x_i$ is a nonnegative integer.

Due to the stochasticity of the processes in each particular system, the observable quantities are computed using the averaging over a large time period, or a large ensemble of systems.
As an alternative to the Monte-Carlo-type stochastic simulation (SSA, \cite{gillespie-ssa-1976}), one may consider a \emph{deterministic} difference equation on the joint probability density, the chemical master equation \cite{vankampen-stochastic-1981}, which reads
\begin{equation}
\dfrac{dP(\mathbf{x},t)}{dt} = \sum\limits_{m=1}^M w^m(\mathbf{x}-\mathbf{z^m}) P(\mathbf{x}-\mathbf{z^m},t) - w^m(\mathbf{x}) P(\mathbf{x},t),
\label{cme:eqn:cme}
\end{equation}
where
$$
P(\mathbf{x}, t): (\{0\} \cup \mathbb{N})^d \cup [0,T] \rightarrow \mathbb{R}_+
$$
is the joint probability of the numbers of molecules of the species to take particular values $x_1,...,x_d$ at the time $t$.

The operator $A$ in \eqref{cme:eqn:cme} appears as an infinite matrix, and $P$ is represented by an infinite vector.
However (at least for a finite time), very large copy numbers are unlikely to appear,
$$
P(\mathbf{x},t) \rightarrow 0, ~|\mathbf{x}| \rightarrow \infty.
$$

The latter property gives a natural way to approximate \eqref{cme:eqn:cme} by a finite problem to make it computationally tractable.
Namely, one employs the finite state space projection (FSP) algorithm \cite{munsky-fsp-2006}:
each $x_i$ is considered in a finite range, $x_i=0,...,N_i-1$.
Note that the probability normalization is not conserved in the truncated equation.
However, the solution deficiency is directly related to the magnitude of the truncated elements.
Choosing $N_i$ large enough, such that $P(\mathbf{x},t)$ is e.g. below the machine precision for $x_i>N_i$, one can neglect the error introduced by the state space truncation.

However, even if each $N_i = \mathcal{O}(N)$ is of the order of tens, the total number of degrees of freedom scales as $N^d$, which becomes infeasible rapidly with increasing number of species $d$.
In this paper, we will use the approximation techniques based on the separation of variables, to represent $N^d$ elements \emph{indirectly} by a much smaller amount of data.

\section{Tensor formats and methods}
\subsection{Separation of variables principles and notations}
A $d$-variate function $f$ is called separable, if it is representable in the form
\begin{equation}
f(x_1,\ldots,x_d) = f^{(1)}(x_1) f^{(2)}(x_2) \cdots f^{(d)}(x_d).
\label{eq:rank1sep}
\end{equation}
Note that we do not specify the domain for $\mathbf{x}$ here.
For example, it may be the whole space $\mathbb{R}^d$, or a finite hypercube, $\{0,\ldots,N-1\}^{\otimes d}$.
The latter case appears in particular in the truncated CME model.
A motivation to the idea of using tensor format approximations in the Chemical Master or Fokker-Planck equations comes from the following observation: if $x_1,\ldots,x_d$ are independent random quantities (i.e. there are no reactions between different species), the joint probability density function equals to the product of  marginal probabilities, $P(\mathbf{x}) = P^{(1)}(x_1) \cdots P^{(d)}(x_d)$.
Note that this is exactly the case of a separable function.
Like a system with all independent species is an ultimate degenerate case from the physical point of view, the separability property is a basic element of tensor product formats, and we may expect ``weakly'' coupled systems to be well-described by functions in low-parametric representations.

First of all, let us see why the separable representation is good from computational point of view.
Starting from this moment, we will consider only a finite domain, $x_i = 0,\ldots,N_i-1$.
If initially the function was defined on a continuous set, we assume that some grid-based discretization is applied.
With this assumption, a function $f$ is fully described by its values $\{f(\mathbf{x})\}$.
In turn, this set of values is stored in the computer memory as a $d$-dimensional array, or \emph{tensor}.
Armed with this duality, by ``tensor'' we may also mean the function it is sampled from, and vice versa.
The memory consumption to store such a tensor is $\prod_{i=1}^d N_i \le N^d$.
This exponential growth with $d$ is called the \emph{curse of dimensionality}.
However, since the separable representation is fully defined by its \emph{factors} $f^{(i)}(x_i)$, $i=1,\ldots,d$, one needs to store only $\sum_{i=1}^d N_i \le d N \ll N^d$ values.
What is more important, operations like computation of a prescribed element, summation (integration via quadrature formulas), etc. are recast directly to the factor elements, and are cheap as well.

Unfortunately, this form of separability is a very rare case in practice.
However, a working idea is to approximate a function with a certain \emph{sum} of separable ones.
A very straightforward approach is the direct sum of several \emph{components}:
$$
f(x_1,\ldots,x_d) = \sum\limits_{\alpha=1}^R f^{(1)}_{\alpha}(x_1) f^{(2)}_{\alpha}(x_2) \cdots f^{(d)}_{\alpha}(x_d).
$$
This representation is the so-called \emph{canonical rank-$R$} tensor format, also known as CANDECOMP and PARAFAC.
It is traditional and commonly used in multilinear algebra and numerical analysis since \cite{hitchcock-rank-1927}, see the following surveys and lecture notes \cite{kolda-review-2009,khor-survey-2011,khor-lectures-2010,hackbusch-2012,larskres-survey-2013}.
The canonical format requires $\mathcal{O}(dNR)$ values to store, and as soon as the \emph{tensor rank} $R$ can be chosen small, provides a substantial reduction of degrees of freedom.
The well-known drawback is that it is difficult to approximate a given function in the canonical format, due to its intrinsic uncloseness, and ill-posedness of the optimization problem on factor elements in the case $R>1$.
As a remedy, one may consider the greedy approach, by looking for consecutive rank-1 corrections (stable), see \cite{Cances-greedy-2011,Ammar-cme-2011,maday-greedy-2009} and references therein.
However, the quality of rank-1 residual approximations deteriorates rapidly with iterations, and as a result, the greedy method may stagnate at rather high levels of solution error.

To stay withing the formats which form closed manifolds, and hence make the approximation problem well-posed,
one considers the representations based on the low-rank decomposition of a matrix.
A traditional approach in this class is the Tucker format \cite{tucker}:
\begin{equation}
f(x_1,\ldots,x_d) = \sum\limits_{\gamma_1=1}^{r_1} \cdots \sum\limits_{\gamma_d=1}^{r_d} G_{\gamma_1,\ldots,\gamma_d} f^{(1)}_{\gamma_1}(x_1) f^{(2)}_{\gamma_2}(x_2) \cdots f^{(d)}_{\gamma_d}(x_d),
\label{eq:tucker}
\end{equation}
%
However, the Tucker \emph{core} $G$ contains $\prod r_i \le r^d$ elements, i.e. still suffers from the curse of dimensionality.
Though it may be quite efficient for low $d$ and $r$, e.g. in three-dimensional problems \cite{khor-chem-2011,khor-ml-2009,meyer-mctdh-book-2009}, in higher dimensions there exist more robust techniques.

One of the most simplest but powerful decompositions is the Tensor Train (TT), or Matrix Product States (MPS) format.
The Matrix Product States representation has appeared in the quantum physics community \cite{ostlund-dmrg-1995,white-dmrg-1993}, along the line with some algorithms.
For example, the so-called Density Matrix Renormalization Group (DMRG) \cite{ostlund-dmrg-1995,white-dmrg-1993} is a numerical variational technique in the MPS format devised to obtain the ground states of spin systems with high accuracy.
It traces back to \cite{white-dmrg-1993,ostlund-dmrg-1995}, and it is nowadays the most efficient method for one-dimensional quantum systems.
In the community of numerical analysis, the MPS was reopened in 2009, as
the so-called \emph{tensor train format}, or simply TT format \cite{ot-tt-2009,osel-tt-2011}.

As well as the canonical and Tucker representations, the TT format exploits the summation of separable components.
A crucial thing is the \emph{structure} of such summation.
\begin{definition}
Given a function $f(\mathbf{x})$, the \emph{Tensor Train representation} reads
\begin{equation}\label{par:tt1}
f(x_1,\ldots,x_d) = \sum\limits_{\alpha_1=1}^{r_1} \cdots \sum\limits_{\alpha_{d-1}=1}^{r_{d-1}} f^{(1)}_{\alpha_1}(x_1) f^{(2)}_{\alpha_1,\alpha_2}(x_2) \cdots f^{(d-1)}_{\alpha_{d-2},\alpha_{d-1}}(x_{d-1}) f^{(d)}_{\alpha_{d-1}}(x_d),
\end{equation}
where $\bm{\alpha} = (\alpha_1,\ldots,\alpha_{d-1})$, $\bm{r}=(r_1,\ldots,r_{d-1})$ are the so-called \emph{TT ranks}, and $f^{(i)}$ are the \emph{TT factors}.
\end{definition}
Note that to store each TT factor in a computer memory, one needs only a three-dimensional array with sizes $r_{i-1} \times N_i \times r_i$ (we agree that $r_0=r_d=1$ for generality).
As soon as $r_i$ can be chosen bounded by a moderate constant, $r_i\le r$, the total storage estimates as $\mathcal{O}(dNr^2)$, and the reduction w.r.t. $N^d$ is significant.
Considering each factor as a matrix-valued function on $x_i$, i.e. taking $\alpha_{i-1},\alpha_i$ as row and column indices, respectively, one may write a matrix product counterpart of \eqref{par:tt1} (hence the name Matrix Product States),
\begin{equation}
f(x_1,\ldots,x_d) = f^{(1)}(x_1) \cdots f^{(d)}(x_d), \quad f^{(i)}(x_i) \in \mathbb{R}^{r_{i-1} \times r_i}.
\label{eq:tt2}
\end{equation}
Clearly, the separability \eqref{eq:rank1sep} is a particular case of any format with all $r_i=1$.

The term TT rank has a strong algebraic meaning, namely, it is nothing else than the matrix rank of a \emph{matricization} of a tensor.
\begin{definition}
By a \emph{multiindex} $\mathbf{x} = \overline{x_1,\ldots,x_k}$ we denote an index which takes all possible combination of values of $x_1,\ldots,x_k$, i.e. if $x_i=1,\ldots,N_i$, then
$$
\overline{x_1,\ldots,x_k} = x_1 + (x_2-1)N_1 + \cdots + (x_k-1) N_1 \cdots N_{k-1}.
$$
\end{definition}
Now we may introduce the matricization.
\begin{definition}
Tensor entries are said to be arranged into a $i$-th \emph{unfolding matrix} $f^{\{i\}}$, if
$$
f^{\{i\}} = \begin{bmatrix}f^{\{i\}}_{\overline{x_1,\ldots,x_i},\overline{x_{i+1},\ldots,x_d}}\end{bmatrix} \in \mathbb{R}^{N_1\cdots N_i \times N_{i+1} \cdots N_d}, \quad f^{\{i\}}_{\overline{x_1,\ldots,x_i},\overline{x_{i+1},\ldots,x_d}}=f(x_1,\ldots,x_d),
$$
i.e. $\overline{x_1,\ldots,x_i}=1,\ldots,N_1\cdots N_i$ is a row index, and $\overline{x_{i+1},\ldots,x_d}=1,\ldots,N_{i+1} \cdots N_d$ is a column index.
\end{definition}
As is shown in \cite{ot-tt-2009,osel-tt-2011}, for a TT rank it holds $r_i = \mathrm{rank}(f^{\{i\}})$.
Therefore, for a fixed $\bm{r} = (r_1,\ldots,r_{d-1})$, the multilinear representation \eqref{par:tt1} defines a closed embedded manifold
$\mathbb{TT}_{\bm{r}}$ in the linear space of all $d$-tensors.
What is important, the matrix routines (like QR and SVD decompositions) of each unfolding matrix are recast to the corresponding operations on factors.
The main consequence is a fast robust \emph{TT truncation}, or \emph{rounding} procedure: given a tensor $f$ in the TT format, find a TT approximation $\tilde f$ with the prescribed accuracy threshold $\varepsilon$ and smallest possible TT ranks.
The asymptotic complexity of the TT rounding is $\mathcal{O}(dNr^3)$, i.e. free from the curse of dimensionality.
Moreover, algebraic operations such as additions, pointwise or scalar products over the elements of initial tensors (vectors of size $N^d$) are recast to operations with factors.
Taking into account the similarity of the TT rounding with the standard rounding of a number, one may think of the \emph{tensor arithmetics}, see \cite{osel-tt-2011} for details.

Discrete operators acting on vectors of size $N^d$ deserve specific consideration.
The multidimensional matrix-by-vector product reads
\begin{equation}
g(\mathbf{x}) = \sum\limits_{\mathbf{x'}} A(\mathbf{x},\mathbf{x'}) f(\mathbf{x'}), \qquad \mathbf{x}=(x_1,\ldots,x_d), \quad \mathbf{x'} = (x'_1,\ldots,x'_d).
\label{eq:tt-mv}
\end{equation}
However, $A$ as a matrix with indices $\mathbf{x},\mathbf{x'}$ is usually of full rank, and its straightforward TT representation $A^{(1)}(x_1) \cdots A^{(2d)}(x_d')$ is inefficient, since $r_d = N^d$.
Therefore, the \emph{matrix} TT format is introduced using the permutation of variables:
\begin{equation}
A(x_1,\ldots,x_d, x'_1,\ldots,x'_d) = \sum\limits_{\beta_1=1}^{R_{1}} \cdots \sum\limits_{\beta_{d-1}=1}^{R_{d-1}} A^{(1)}_{\beta_1}(x_1,x'_1) A^{(2)}_{\beta_1,\beta_2}(x_2,x'_2) \cdots A^{(d)}_{\beta_{d-1}}(x_d,x'_d).
\label{eq:ttm}
\end{equation}
Then the matrix-by-vector product \eqref{eq:tt-mv} is written in the TT format as follows,
$$
g(x_1,\ldots,x_d) = g^{(1)}(x_1) \cdots g^{(d)}(x_d), \quad g^{(i)}_{\overline{\alpha_{i-1},\beta_{i-1}},\overline{\alpha_i,\beta_i}}(x_i) = \sum\limits_{x'_i} A^{(i)}_{\beta_{i-1},\beta_i}(x_i,x'_i) f^{(i)}_{\alpha_{i-1},\alpha_i}(x'_i).
$$
Moreover, the standard Kronecker product of matrices $C=A \otimes B$ means elementwise $C(\overline{xy},\overline{x'y'})= A(x,x') B(y,y')$, i.e. is also a degenerate case of the matrix TT format with rank 1.

In the next section we will consider explicit TT structures of CME operators and the corresponding TT ranks.
So, we need to distinguish the TT ranks (ranks of unfolding matrices) from the ordinary rank of a matrix in the initial variables.
\begin{definition}
By \emph{tensor rank} we denote a tuple of TT ranks,
$$
\mathrm{\mathbf{trank}}(f) = (r_1,\ldots,r_{d-1}), \qquad \mathrm{\mathbf{trank}}(A)=(R_1,\ldots,R_{d-1}),
$$
where the numbers $r_i$, $i=1,\ldots,d-1$ are according to \eqref{par:tt1}, and $R_i$ are according to \eqref{eq:ttm}.
Asymptotic estimates are written w.r.t. the maximal TT rank, i.e.
$$
\mathrm{trank}(\cdot) = \max \mathrm{\mathbf{trank}}(\cdot).
$$
\end{definition}

This definition is not limited to the TT format only.
In a similar way, any tensor format may be characterized by its tensor rank tuple.
For example, another alternative to the Tucker and TT formats might be the $\mathcal{H}$T format \cite{hackbusch-2012}.
It also has an analog in the physics community, the so-called \emph{Tensor Tree Networks} (TTN) representation \cite{fannes-TTN-1992}.

\subsection{Virtual dimensions and quantization}
We see that, as soon as the data is ``structured'' well enough, and TT ranks $\mathrm{\mathbf{trank}}(f)$ are moderate and do not depend on $N$ and $d$, the complexity of TT operations is linear in $d$.
This fact motivated the development of the so-called \emph{Quantized-TT} (QTT) format \cite{osel-2d2d-2010,khor-qtt-2011}, exploited heavily in this paper as well.
The idea is as follows.
Suppose the number of admissible values for each $x_i$ is a power of $2$, i.e. $N_i=2^L$.
Then they can be enumerated via the binary coding,
$$
x_i = \sum\limits_{l=1}^L x_{i,l} \cdot 2^{l-1}, \quad x_{i,l} \in \{0,1\}.
$$
It corresponds to the reshaping of a tensor into a $D=dL$-dimensional one, for which the TT decomposition is applied.
If the TT ranks of this $D$-tensor (or \emph{QTT ranks}) are small, then the storage is logarithmic, $\mathcal{O}(dL\cdot 2 \cdot \mathrm{trank}^2) = \mathcal{O}(d \log N)$.

The following 1D functions are exactly well-representable in the QTT format:
\begin{enumerate}
 \item exponential $\exp(\kappa x) = \mathrm{e}^{(1)}(x_1) \cdots \mathrm{e}^{(L)}(x_L), \quad \mathrm{e}^{(l)}(x_l) = \exp(\kappa x_l \cdot 2^{l-1})$ (QTT rank 1),
 \item trigonometric $\sin(\kappa x+\phi)$ (QTT rank 2) and polynomial $x^p$ (QTT rank $p+1$),
 \item delta-function $\delta(x-a) = \delta^{(1)}(x_1) \cdots \delta^{(L)}(x_L), \quad \delta^{(l)}(x_l) = \delta(x_l-a_l),~a= \sum\limits_{l=1}^L a_{l} \cdot 2^{l-1}$.
\end{enumerate}
The latter example often arises as the initial state in the Chemical Master Equation.
Lots of other smooth functions in 1D do not have exact QTT representations, but still admit highly accurate QTT approximations with rather small ranks, since the standard theory of polynomial interpolation may be used (e.g. Gaussian $\exp(-x^2)$, or Michaelis-Menten $1/(1+x)$ functions on uniform grids are numerically approximated with QTT ranks bounded by $8$ up to the accuracy $10^{-10}$).

Moreover, the QTT format allows simple constructive representations of basic operators (Laplace, gradient or shift)  \cite{khkaz-lap-2012} on  uniform tensor grids.

The logarithmic complexity w.r.t. the cardinality of the initial tensor makes the QTT format a very promising tool for high-dimensional problems.
There are algorithms, which rely essentially on the binary QTT structure,
such as the super-fast data-sparse Fourier transform \cite{dks-ttfft-2012}.
One may also consider the time as an additional dimension with its further decomposition via the QTT format \cite{DKhOs-parabolic1-2012,gavkh-qtt_cayley-2011}.

In higher dimensions the functions are usually assembled using the Kronecker products, i.e. in the rank-1 form \eqref{eq:rank1sep}, keeping in mind that each factor $f^{(i)}(x_i)$ is in turn the QTT format.
For example, the propensity functions in the CME often depend on only a few variables, say $w(\mathbf{x}) = \hat w(x_m)$.
Such functions have ``perfect'' rank-1 representations,
$$
w(\mathbf{x}) = w^{(1)}(x_1) \cdots w^{(d)}(x_d), \quad w^{(m)}(x_m) = \hat w(x_m), \quad w^{(i)}(x_i)=1,~i \neq m.
$$

However, this might be not the case, and the ranks of the straightforward QTT representation described above (referred later as \emph{linear QTT}) may grow rapidly with the accuracy.
For example, a pattern of the function $f(x_1,x_2)=\exp(-(x_1-x_2)^2+x_2^2)$ looks similar to a diagonal matrix.
Obviously, the TT rank of its approximation will be rather high, qualitatively proportional to the number of sufficiently large entries on the diagonal.

To relax this problem, a combined tensor representation was proposed, called the \emph{QTT-Tucker} format \cite{dk-qtt-tucker-2013}.
The initial variables are separated in the Tucker way \eqref{eq:tucker}.
To treat the curse of dimensionality, the Tucker core is stored in the TT format,
$$
G_{\gamma_1,\ldots,\gamma_d} = \sum\limits_{\alpha_1,\ldots,\alpha_{d-1}} G^{(1)}_{\alpha_1}(\gamma_1) G^{(2)}_{\alpha_1,\alpha_2}(\gamma_2) \cdots G^{(d)}_{\alpha_{d-1}}(\gamma_d).
$$
Finally, the Tucker factors are compressed in the QTT format,
$$
f^{(i)}_{\gamma_i}(x_i) = \sum\limits_{\gamma_{i,1},\ldots,\gamma_{i,L-1}} f^{(i,1)}_{\gamma_i,\gamma_{i,1}}(x_{i,1}) f^{(i,2)}_{\gamma_{i,1},\gamma_{i,2}}(x_{i,2}) \cdots f^{(i,L)}_{\gamma_{i,L-1}}(x_{i,L}).
$$
As was demonstrated in \cite{dk-qtt-tucker-2013}, the function-related tensors may require smaller QTT-Tucker ranks for the same accuracy, than the linear QTT ones.
However, the asymptotic storage complexity of the $\mathcal{H}$T or QTT-Tucker formats is cubic w.r.t. the tensor rank, while in the linear QTT it is quadratic, and it is hard to conclude a priory which format will perform better in each particular example.
On the other hand, since the QTT-Tucker is a certain combination of the TT decompositions, lots of TT algorithms may be simply reused.

Despite that all objects involved in the CME are high-dimensional, and hence are ``tensors'', we find it convenient to distinguish TT representations of ``matrix'' (``operator'')  \eqref{eq:ttm} and ``vector'' (``tensor'') \eqref{par:tt1}, especially because it allows to refer to simple algebra in full tensor spaces, like matrix-by-vector product, etc.

\subsection{Computational tensor methods}\label{sec:amrs}
The variety of tensor operations is not limited to a basic multilinear algebra described up to now.
To successfully solve high-dimensional equations (in particular, the Chemical Master Equation), one needs more advanced techniques: solvers of linear systems (stationary problems and implicit time schemes), eigenvalue problems, more sophisticated time integrators (like matrix exponential), etc.
The first attempts were connected with the usage of traditional iterative methods (e.g. CG, Arnoldi) equipped with the tensor arithmetics \cite{balgras-Htuck_gmres-2013,lebedeva-tensorcg-2010,TobKress-Krylov-2010,dc-tt_gmres-2013}.
However, these methods suffer from one important drawback: in addition to the solution and right-hand side, they require auxiliary vectors, which also have to be approximated in the tensor format.
As was observed in many examples (see e.g. \cite{dc-tt_gmres-2013} for the GMRES method), even if the solution may be well representable in the format, the residual and Krylov vectors manifest significantly larger tensor ranks, especially in latter iterations.
This introduces a significant overhead in the computational complexity.

An alternative approach, which is free from expensive auxiliary tensors, is the alternating variational methods.
The multilinearity of the considered tensor formats with respect to the elements of each factor allows to construct efficient iterative model reduction techniques.
A brief sketch of the alternating iterations is formulated as follows.
Given a quadratic function $J(f)$ to optimize and an initial guess $f$ in a tensor format, fix all the factors of $f$ except one, and optimize $J$ over the elements of that chosen factor.
\begin{algorithm}[t]
 \caption{ALS iteration} \label{alg:als}
 \begin{algorithmic}[1]
  \REQUIRE Function $J(f)$ and initial guess $f$ in the TT format~\eqref{eq:tt2}.
  \ENSURE Updated solution $\hat f$.
  \FOR[Cycle over TT cores]{$i=1,\ldots,d$}
    \STATE Find $\hat f^{(i)} = \arg\min_{f^{(i)}_{*}} J\left(\hat f^{(1)}(x_1) \cdots \hat f^{(i-1)}(x_{i-1}) \cdot f^{(i)}_{*}(x_i) \cdot f^{(i+1)}(x_{i+1}) \cdots f^{(d)}(x_d)\right)$
  \ENDFOR
  \RETURN $\hat f =  \hat f^{(1)}(x_1) \cdots \hat f^{(d)}(x_d)$.
 \end{algorithmic}
\end{algorithm}
We obtain the so-called Alternating Linear Scheme (ALS), one iteration in the TT format is shown in Alg. \ref{alg:als}.
Since $f$ depends linearly on $f^{(i)}$, the reduced (constrained) function at line 2 of Alg. \ref{alg:als} remains quadratic, and can be optimized efficiently.

Three types of functions are especially important for us: error $J(f) = \|f-g\|^2$ (corresponds to the low-rank approximation problem), energy $J(f) = (f,Af)-2(f,g)$ and residual $J(f) = \|Af-g\|^2$ (solution of linear systems).
The alternating approximation techniques have been developing since \cite{harshman-parafac-1970,beylkin-high-2005,lathauwer-rank1-2000} for the canonical and Tucker formats under the name Alternating \emph{Least Squares}.
More general description extended to linear system and eigenvalue problems (as well as the name Alternating \emph{Linear Scheme}) was provided in \cite{holtz-ALS-DMRG-2012}, including also a connection between the Tensor Train and Matrix Product States formats, and the DMRG and ALS methods.

However, the ALS Algorithm \ref{alg:als} is known to suffer from several drawbacks.
First, it does not allow to adapt tensor ranks during the iterations, and one has to guess them a priory, which is a difficult problem.
Second, the convergence may be extremely slow, and moreover, the method is likely to get stuck in spurious local minima, since the cost function $J$ is nonconvex w.r.t. the elements of \emph{all} TT factors.
Finally, even the reduced problem might still be large and/or ill-conditioned.

Fortunately, the cornerstone idea of alternating model reduction appeared to be very efficient after certain modifications.
The first cycle of papers \cite{holtz-ALS-DMRG-2012,DoOs-dmrg-solve-2011,Os-mvk2-2011,so-dmrgi-2011proc,tobler-ht_dmrg-2011} is devoted to further improvements of the so-called \emph{two-sited DMRG}.
The main difference with Alg. \ref{alg:als} is that the target function is constrained to two TT factors, \\
$
\min_{f^{(i)}_{*}, f^{(i+1)}_{*}} J\left(\hat f^{(1)}(x_1) \cdots \hat f^{(i-1)}(x_{i-1}) \cdot f^{(i)}_{*}(x_i) f^{(i+1)}_{*} (x_{i+1})  \cdot f^{(i+2)}(x_{i+2}) \cdots f^{(d)}(x_d)\right),
$
instead of one.
This allows to include the rank adaptation step, and accelerate the convergence in some cases.
However, the issue of stagnation in local minima still appears, especially in high-dimensional problems.

A crucial remedy was to combine the ALS method with the classical iterative steps \cite{ds-amr1-2013,ds-amr2-2013}.
The so-called Alternating Minimal Energy (AMEn) algorithm eliminates disadvantages of both approaches.
In each step, the solution is updated according to the ALS scheme, but in the next step, the TT representation of the solution is \emph{enriched} by adding a partial TT format of the residual, similarly to the Steepest Descent, or FOM(1) method.
Therefore, in the forthcoming updates, the system is projected onto a wider basis, containing both  the current solution and residual components.
As a result, a global convergence rate may be estimated using the Steepest Descent theory,
although the practically observed convergence is significantly faster, and allows to treat efficiently even non-symmetric linear systems.

As was found, even a very crude TT approximation to the residual is enough for the method to converge efficiently.
Therefore, to compute it cheaply, we may employ a simple \emph{second} ALS iteration for the TT approximation purpose (AMEn+ALS according to \cite{ds-amr2-2013}).
The whole technique is summarized in Alg. \ref{alg:amen} (for nonsymmetric systems, the minimization of $J(f)$ is formally replaced by the Galerkin linear system arising from $\partial J/\partial f = 0$).
It is the method that will be used in our numerical experiments to solve the linear systems arising from implicit time schemes.

\begin{algorithm}[t] 
 \caption{AMEn iteration (for brevity, $a^{(m)}$ denotes $a^{(m)}(x_m)$, $a \in\{f,\hat f,z,\ldots\}$)} \label{alg:amen}
 \begin{algorithmic}[1]
  \REQUIRE Function $J(f)=(f,Af)-2(f,g)$, initial guesses $f$, $z$ in the TT format.
  \ENSURE Updated solution $\hat f$, residual $\hat z$.
  \FOR[Cycle over TT cores]{$i=1,\ldots,d$}
    \STATE Find $\hat f^{(i)} = \arg\min_{f^{(i)}_{*}} J\left(\hat f^{(1)} \cdots \hat f^{(i-1)} \cdot f^{(i)}_{*} \cdot f^{(i+1)} \cdots f^{(d)}\right)$.
    \STATE Reduce $\mathrm{\bf{trank}}_{i}(\hat f)$, \quad $\hat f = \hat f^{(1)} \cdots \hat f^{(i)} \cdot f^{(i+1)} \cdots f^{(d)}$ \COMMENT{Optionally}
    \STATE Update the residual \\
    $\hat z^{(i)} = \arg\min_{z^{(i)}_{*}} \left\|(A \hat f - g) - \left(\hat z^{(1)} \cdots \hat z^{(i-1)} \cdot z^{(i)}_{*} \cdot z^{(i+1)} \cdots z^{(d)}\right)\right\|^2$.
    \IF{$i<d$}
    \STATE Compute the enrichment \\
    $\hat s^{(i)} = \arg\min_{s^{(i)}_{*}} \left\|(A \hat f - g) - \left(\hat f^{(1)} \cdots \hat f^{(i-1)} \cdot s^{(i)}_{*} \cdot z^{(i+1)} \cdots z^{(d)}\right)\right\|^2$.
    \STATE Expand the basis $\hat f^{(i)}=\begin{bmatrix}\hat f^{(i)} & \hat s^{(i)}\end{bmatrix}.$
    \ENDIF
  \ENDFOR
 \end{algorithmic}
\end{algorithm}

The following lemma estimates the computational complexity.
\begin{lemma} \cite{holtz-ALS-DMRG-2012,DoOs-dmrg-solve-2011,ds-amr1-2013}\label{lem:als_complexity}
Suppose that a linear system $Af=g$ of size $N^d$ is given in the TT format, with $\mathrm{trank}(A) \leq r_A$, $\mathrm{trank}(g) \leq r_g$, and the solution ranks are kept $\mathrm{trank}(f) \leq r_f$.
Then, one iteration of the ALS or AMEn methods requires $\mathcal{O}(d N r_f^2 + d r_f^2 r_A + d r_f r_g)$
memory, and $\mathcal{O}(d N r_f^3 r_A + d N^2 r_f^2 r^2_A + d N r_f^2 r_g)$ operations.
\end{lemma}
Weak places in this lemma are the total number of iterations and the actual rank bound $r_f$ of the solution.
In general, it is difficult to predict tensor ranks of e.g. CME solution a priori.
The number of iterations may be determined from the spectrum of the matrix, but in most cases it will be far larger than the one is really required.
Fortunately, the AMEn method does not tend to overestimate the quasi-optimal tensor rank provided by the TT rounding procedure at a given accuracy, which appears to be moderate in the numerical examples considered.

Finally, we note that the alternating methods for the QTT-Tucker format are constructed naturally from their corresponding TT versions as the building blocks (see \cite{dk-qtt-tucker-2013} for details).

\section{Tensor representation of typical CME operators}
In the following, we will use a more convenient counterpart to \eqref{cme:eqn:cme} with the help of the shift matrices.
Denote
\begin{equation}
J^{z} = \begin{bmatrix}
         0 & \cdots & 1 \\
           & \ddots & & \ddots \\
           & & \ddots & & 1 \\
           & & & \ddots & \vdots \\
         & &  &  & 0
        \end{bmatrix} \begin{matrix} \phantom{0} \\ \phantom{\ddots} \\ \leftarrow & \mbox{row} & N-z \\ \phantom{\ddots} \\ \leftarrow & \mbox{row} & N\phantom{-z,} \end{matrix},
        \quad \mbox{if}~z \ge 0,
\label{cme:eqn:shift_matrices}
\end{equation}
and for $z<0$ we agree $J^{z} = (J^{-z})^\top$.
Now we write the finite state approximation (FSP) of \eqref{cme:eqn:cme} as a linear ODE,
\begin{equation}
\dfrac{dP(t)}{dt} = AP(t), \qquad A=\sum\limits_{m=1}^M (\mathbf{J}^{\mathbf{z^m}} - \mathbf{J}^0) \diag(w^m) P(t), \qquad P(t) \in \mathbb{R}_+^{\prod_{i=1}^d N_i},
\label{cme:eqn:cme-shift}
\end{equation}
where the multidimensional shift operator reads
\begin{equation*}
 \mathbf{J}^{\mathbf{z}} = J^{z_1} \otimes \cdots \otimes J^{z_d}, \quad \mathbf{J}^{\mathbf{z}}(x_1,\ldots,x_d,x_1',\ldots,x_d') = J^{z_1}(x_1,x_1')  \cdots J^{z_d}(x_d,x_d'),
\end{equation*}
$w^m = \{w^m(\mathbf{x})\}$ and $P(t) = \{P(\mathbf{x},t)\}$, $\mathbf{x} \in \bigotimes\limits_{i=1}^d \{0,\ldots,N_i-1\}$, are the corresponding values of $w^m$ and $P$ stacked into vectors,
$\diag(w^m)$ is a diagonal matrix with the values of $w^m$ stretched along the diagonal, and $\otimes$ means the Kronecker product, or a rank-1 matrix TT format.
Note that $\mathbf{J}^0$ is just an identity matrix of proper sizes.

To employ tensor decompositions, we need to present all initial data in our favorable format.
Assuming the tensor separability of each propensity function $w^{m}$, we obtain immediately from \eqref{cme:eqn:cme-shift} the tensor rank estimate of the CME operator.
\begin{lemma}\label{lem:rank-A-gen}
\begin{equation}
\mathrm{trank}(A) \le \sum_{m=1}^M 2 \cdot \mathrm{trank}(w^{m}).
\label{cme:eqn:rank-A-gen}
\end{equation}
\end{lemma}
\begin{proof}
Both $\mathbf{J}^{\mathbf{z^m}}$ and $\mathbf{J}^{0}$ are (rank-1) Kronecker products.
The difference $\mathbf{J}^{\mathbf{z^m}} - \mathbf{J}^0$ is therefore of rank 2.
Given a separable form $w^m = \sum_{\bm{\alpha}} W^{(1)}_{\alpha_1} \otimes W^{(2)}_{\alpha_1,\alpha_2} \otimes \cdots \otimes W^{(d)}_{\alpha_{d-1}}$, the matrix $\diag(w^m)$ is constructed without changing the rank,
$$
\diag(w^m) = \sum_{\bm{\alpha}} \diag(W^{(1)}_{\alpha_1}) \otimes \diag(W^{(2)}_{\alpha_1,\alpha_2}) \otimes \cdots \otimes \diag(W^{(d)}_{\alpha_{d-1}}).
$$
The product of $\mathbf{J}^{\mathbf{z^m}} - \mathbf{J}^0$ and $\diag(w^m)$ multiplies the ranks, and finally, we sum the terms corresponding to all reactions.
\end{proof}

\begin{remark}
It is also not difficult to estimate tensor ranks of the QTT decomposition of $A$, provided that the QTT ranks of each propensity $w^m(\mathbf{x})$ are limited.
Indeed, each shift matrix $J^{z}$ possesses a rank-2 QTT representation \cite{khkaz-conv-2013}.
Therefore, in the QTT format it holds $\mathrm{trank}(A) \le \sum_{m=1}^M 3 \cdot \mathrm{trank}(w^{m})$.
In turn, if $w^m$ is a direct product of smooth functions, it can be approximated by a degree-$p$ polynomial interpolation, with the rank-$(p+1)$ QTT representation.
In particular, such an approach was applied in \cite{kkns-cme-theory-2013} to certain classes of separable CME models.
\end{remark}

\subsection{Reversible monomolecular reactions} \label{sec:z=1}
A case of special interest is the reactions of form $\emptyset \leftrightarrows S_i$,
that is a simple creation/destruction of specie molecules.
In this case, each $\mathbf{z^m}$ contains only $+1$ or $-1$ at just one position, and $\sum\limits_{m=1}^M z^m_i = 0,~i=1,\ldots,d$.
Such reactions appear frequently in gene regulatory networks, such as switches, cascades, etc. (see below).
We have thus $M=2d$ reactions, and \eqref{cme:eqn:cme-shift} can be separated into two parts, corresponding respectively to the \emph{creation} and \emph{destruction} of specie molecules,
\begin{equation}
\dfrac{dP(t)}{dt} = \sum\limits_{m=1}^d (\mathbf{J}^{m+} - \mathbf{J}^0) \diag(w^{m+}) P(t) + \sum\limits_{m=1}^d (\mathbf{J}^{m-} - \mathbf{J}^0) \diag(w^{m-}) P(t),
\label{cme:eqn:cme_cr_destr}
\end{equation}
$w^{m+}$ is the propensity corresponding to the creation reactions with $\mathbf{z^m} \ge 0$, $w^{m-}$ corresponds to $\mathbf{z^m} \le 0$, and
$$
\mathbf{J}^{m+} = \underbrace{J^0 \otimes \cdots \otimes J^{-1}}_{m} \otimes J^0 \cdots \otimes J^0, \quad
\mathbf{J}^{m-} = \underbrace{J^0 \otimes \cdots \otimes J^{1}}_{m} \otimes J^0 \cdots \otimes J^0.
$$
Thus the shift operators in each part can be collected into separable difference operators, acting only on $x_m$,
\begin{equation}
 \nabla^{-}_m = \underbrace{J^0 \otimes \cdots \otimes (J^0-J^{-1})}_{m} \otimes J^0 \cdots \otimes J^0, \quad
 \nabla^{+}_m = \underbrace{J^0 \otimes \cdots \otimes (J^{1}-J^0)}_{m} \otimes J^0 \cdots \otimes J^0.
 \label{cme:eqn:grads}
\end{equation}
Now, the CME operator $A$ \eqref{cme:eqn:cme-shift} factorizes as
\begin{equation}
A = A^+ + A^-, \qquad
A^+ = -\sum\limits_{m=1}^d \nabla_m^{-} \diag(w^{m+}), \quad A^- = \sum\limits_{m=1}^d \nabla^{+}_m \diag(w^{m-}),
 \label{cme:eqn:cme_grads}
\end{equation}
which has a very close form to the diffusion-convection equation discretized using the finite difference scheme.

The general Lemma  \ref{lem:rank-A-gen} can be applied, giving the tensor rank bound proportional to $d$.
However, we would like to consider typical gene networks in more detail and prove refined results.

\subsection{Signaling cascade genetic model}
A cascade process occurs when adjacent genes produce proteins which influence on the expression of a succeeding gene, see Fig. \ref{fig:cascade}.
This is a typical model in genetic networks; as an example, the lytic phase of the $\lambda$-phage system \cite{ptashne-switch-1992} can be considered.
A mutually repressing gene pair, or gene toggle (Fig. \ref{fig:toggle}), that can be found in such systems, is also a case of the cascade model, with the dimension $2$.
\begin{figure}[h!]
\begin{minipage}[t]{0.49\linewidth}
\centering
\caption{Cascade signaling network}
\label{fig:cascade}
\begin{tikzpicture}
\node [circle, draw] (s1) at (0,0) {$S_1$};
\node [circle, draw] (s2) at ($(s1)+(1.5,0)$) {$S_2$};
\node  (ccc) at ($(s2)+(1.2,0)$) {$\cdots$};
\node [circle, draw] (sd) at ($(ccc)+(1.2,0)$) {$S_d$};

\draw[sloped,semithick,->] (s1) to [out=60,in=120] (s2);
\draw[sloped,semithick,->] (s2) to [out=60,in=120] (ccc);
\draw[sloped,semithick,->] (ccc) to [out=60,in=120] (sd);
\draw[sloped,semithick,->] (s1) to [out=-135,in=180] ($(s1)+(0,-0.8)$) to [out=0,in=-45] (s1);
\draw[sloped,semithick,->] (s2) to [out=-135,in=180] ($(s2)+(0,-0.8)$) to [out=0,in=-45] (s2);
\draw[sloped,semithick,->] (sd) to [out=-135,in=180] ($(sd)+(0,-0.8)$) to [out=0,in=-45] (sd);
\end{tikzpicture}
\end{minipage}
\hfill
\begin{minipage}[t]{0.49\linewidth}
\centering
\caption{Toggle switch}
\label{fig:toggle}
\begin{tikzpicture}
\node [circle, draw] (s1) at (0,0) {$S_1$};
\node [circle, draw] (s2) at ($(s1)+(1.5,0)$) {$S_2$};

\draw[sloped,semithick,->] (s1) to [out=60,in=120] (s2);
\draw[sloped,semithick,->] (s2) to [out=-120,in=-60] (s1);
\draw[sloped,semithick,->] (s1) to [out=-180,in=135] ($(s1)+(-0.5,-0.6)$) to [out=-45,in=-90] (s1);
\draw[sloped,semithick,<-] (s2) to [out=-90,in=-135] ($(s2)+(0.5,-0.6)$) to [out=45,in=0] (s2);
\end{tikzpicture}
\end{minipage}
\end{figure}
In Fig. \ref{fig:cascade}, \ref{fig:toggle}, the arrows denote feedbacks between species arising from the propensity rates $w^m$: an arrow going from the $i$-th to $j$-th specie means that the rate of the reaction involving the $j$-th specie depends on the copy number of the $i$-th one.
For example, in the toggle switch, the creation rate of each protein depends inversely proportional on the concentration of another protein, whereas the destruction rate depends only on the concentration of the degrading protein itself.

In such cases, when the $m$-th destruction propensity depends only on $x_m$, its rank-1 decomposition reads
$$
w^{m-}(\mathbf{x}) = e_1(x_1)  \cdots \hat w^{m-}(x_m) \cdots e_d(x_d),
$$
where $e_i(x_i)=1~\forall x_i=0,\ldots,N_i-1$.
Now, the destruction part of the operator, $A^-$ in \eqref{cme:eqn:cme_grads}, has the Laplace-like form
\begin{equation}
A^{-} = D_1 \otimes J^0 \cdots \otimes J^0 + \cdots + J^0 \otimes \cdots \otimes D_d, \quad D_m = (J^1-J^0) \diag(\hat w^{m-}),
\label{eq:A_destr_lapl}
\end{equation}
which is proven to have the TT rank 2 \cite{khkaz-lap-2012}.

The creation part is usually more complicated (it contains feedbacks between species), and depends on several variables.
In the cascade networks (including also the toggle switch), the $m$-th creation propensity depends
on $x_{m-1}$ (or $x_{m+1}$), and, probably, on $x_m$.
Thus, the corresponding operator part sums the two-variate terms,
\begin{equation}
A^{+} = D^1_1 \otimes J^0 \cdots \otimes J^0 + D^2_1 \otimes D^2_2 \otimes J^0 \cdots \otimes J^0 + \cdots + J^0 \cdots \otimes D^d_{d-1} \otimes D^d_d,
\label{eq:A_creat_casc}
\end{equation}
where $D^m_{m-1} = \diag(\hat w^{m+})$, $D^m_m = -(J^0-J^{-1})$ (in a simpler case when $w^{m+}$ does not depend on $x_m$).
For example, the Michaelis-Menten rate in a cascade reads $w^{m+}(\mathbf{x}) = e_1(x_1) \cdots  \hat w^{m+}(x_{m-1}) \cdot e_m(x_m)  \cdots e_d(x_d)$, $\hat w^{m+}(x_{m-1}) = \frac{\alpha}{\beta+x_{m-1}}$.
A generalization will be given in Remark \ref{cme:rem:cascadic_rank>1}.

In the rest of this section, we are going to discuss the structure of TT factors.
Contrarily to the $\mathcal{O}(d)$ canonical rank, provided by Lemma \ref{lem:rank-A-gen}, the following Lemma \ref{cme:lem:cascadic} shows that a sum like $A^{+}$ can be represented via a rank-3 TT decomposition, thus with the linear in $d$ memory cost.
The Matrix Product form \eqref{eq:tt2} is especially convenient for such analysis, but the variables $x_i,x_i'$ would take too much place, and we omit them for brevity, in the same way as in Alg. \ref{alg:amen}.
That is, by $E_i,~F_i^{i},~F^{i+1}_{i}$ we will mean not the whole matrices, but the elements $E_i(x_i,x_i'),~F_i^{i}(x_i,x_i'),~F^{i+1}_{i}(x_i,x_i')$.

\begin{lemma}
\label{cme:lem:cascadic}
Given the matrices $E_i,~F_i^{i},~F^{i+1}_{i}$.
The cascadic sum
\begin{equation}
H(\mathbf{x},\mathbf{x'}) = F^1_1  \left(\prod\limits_{k=2}^{d} E_k \right) + \sum\limits_{i=2}^d \left(\prod\limits_{k=1}^{i-2} E_k \right) \cdot F^i_{i-1} \cdot F^i_i \cdot \left(\prod\limits_{k=i+1}^{d} E_k \right)
\label{cme:eqn:cascadic}
\end{equation}
possesses an explicit exact rank-3 TT decomposition $H(\mathbf{x},\mathbf{x'}) = H^{(1)} \cdots H^{(d)}$, where
\begin{equation}
H^{(1)} = \begin{bmatrix}
     E_1 & F^2_1 & F^1_1
    \end{bmatrix}, \quad
H^{(d)} = \begin{bmatrix}
 0 \\
 F^d_d \\
 E_d
\end{bmatrix}, \quad
H^{(i)} = \begin{bmatrix}
E_i & F^{i+1}_i & 0 \\
0 & 0 & F^i_i \\
0 & 0 & E_i
\end{bmatrix},
\label{cme:eqn:casc_inner_block}
\end{equation}
if $i=2,\ldots,d-1.$
For the Tucker decomposition the same rank-3 bound holds.
\end{lemma}
\begin{proof}
 We begin to split the dimensions recurrently, extracting the linearly independent elements, in the same way as in the TT-SVD algorithm \cite{osel-tt-2011}.
 So, the first step reads
$$
H(\mathbf{x},\mathbf{x'}) = \begin{bmatrix}
     E_1 & F^2_1 & F^1_1
    \end{bmatrix}
    \begin{bmatrix}
     F^3_2  F^3_3  \cdots  E_d + \cdots + E_2  \cdots  F^d_{d-1}  F^d_d\\
     F^2_2  E_3  \cdots  E_d \\
     E_2  \cdots  E_d
    \end{bmatrix}.
$$
The first term here is exactly the first TT factor of the decomposition.
Now, suppose we have the following form,
\begin{equation}
H_i(x_i,\ldots,x_d,x_i',\ldots,x_d') = \begin{bmatrix}
F^{i+1}_i  F^{i+1}_{i+1}  \cdots  E_d + \cdots + E_i  \cdots  F^d_{d-1}  F^d_d \\
F^i_i  E_{i+1}  \cdots  E_d \\
E_i  \cdots  E_d
\end{bmatrix}.
\label{cme:eqn:casc_rec_block}
\end{equation}
We split the $i$-th dimension from each row in the same manner,
$$
H_i(x_i,\ldots,x_d,x_i',\ldots,x_d') = \begin{bmatrix}
E_i & F^{i+1}_i & 0 \\
0 & 0 & F^i_i \\
0 & 0 & E_i
\end{bmatrix}
\begin{bmatrix}
  F^{i+2}_{i+1}  F^{i+2}_{i+2}  \cdots  E_d + \cdots + E_{i+1}  \cdots  F^d_{d-1}  F^d_d \\
  F^{i+1}_{i+1}  E_{i+2}  \cdots  E_d \\
  E_{i+1}  \cdots  E_d
\end{bmatrix},
$$
and derive the $i$-th TT factor (the first term).
The rest dimensions are presented in the same form as \eqref{cme:eqn:casc_rec_block}, so we can continue the process. The last two factors are separated as follows,
$$
\begin{bmatrix}
F^d_{d-1}  F^d_d \\
F^{d-1}_{d-1}  E_d \\
E_{d-1}  E_d
\end{bmatrix} =
\begin{bmatrix}
 F^{d}_{d-1} & 0 \\
 0 & F^{d-1}_{d-1} \\
 0 & E_{d-1}
\end{bmatrix}
\begin{bmatrix}
 F^d_d \\
 E_d
\end{bmatrix} =
\begin{bmatrix}
 E_{d-1} & F^{d}_{d-1} & 0 \\
 0 & 0 & F^{d-1}_{d-1} \\
 0 & 0 & E_{d-1}
\end{bmatrix}
\begin{bmatrix}
 0 \\
 F^d_d \\
 E_d
\end{bmatrix}.
$$
We see that all the TT ranks are equal to 3, which confirms the claim of the lemma.
To obtain the Tucker rank estimate, it is sufficient to note that each TT factor contains only 3 independent elements, and follow the TT-to-Tucker procedure described in \cite{dk-qtt-tucker-2013}.
\end{proof}

\begin{remark}
\label{cme:rem:cascadic_rank>1}
In \eqref{cme:eqn:cascadic}, each summand is a rank-1 tensor. However, we can straightforwardly generalize it to the case, when the neighboring terms are summed from several components,
$$
\sum\limits_{\alpha_i=1}^{r_i} F^i_{i-1,\alpha_i}(x_{i-1},x_{i-1}') \cdot F^i_{i,\alpha_i}(x_i,x_i') \quad  \mbox{instead of} \quad  F^i_{i-1}(x_{i-1},x_{i-1}') \cdot F^i_i(x_{i},x_{i}'),
$$
i.e., the TT rank of each propensity is not equal to 1.
In this case, we can collect respectively the row and column vectors
$$
F^i_{i-1}(x_{i-1},x_{i-1}') = \begin{bmatrix} F^i_{i-1,1}(x_{i-1},x_{i-1}') \\ \vdots \\  F^i_{i-1,r_i}(x_{i-1},x_{i-1}')\end{bmatrix}^\top, \quad
F^i_i(x_i,x_i') = \begin{bmatrix}F^i_{i,1}(x_i,x_i') \\ \vdots \\ F^i_{i,r_i}(x_i,x_i')\end{bmatrix},
$$
and the constructions \eqref{cme:eqn:casc_inner_block} will be considered as block matrices, with
the sizes (i.e. the TT ranks) $(2+r_i) \times (2+r_{i+1})$.
Counting the linearly independent elements in each TT factor, we conclude that the $i$-th Tucker rank is bounded by $1+r_i+r_{i+1}$.
\end{remark}

\begin{corollary}
\label{cor:neigh+diag}
Summing the rank-3 creation part \eqref{eq:A_creat_casc} and the rank-2 destruction part \eqref{eq:A_destr_lapl}, we conclude that the CME operator in the cascade model admits an explicit exact TT decomposition of rank 5.
\end{corollary}
\begin{remark}
Using the technique of linear dependence elimination, employed in Lemma \ref{cme:lem:cascadic}, one can show that in fact $\mathrm{trank}(A)=4$.
\end{remark}
%

\subsection{Application to spin models}
The similar operators arise also in the one-dimensional spin systems modeling with nearest neighbor interactions. For example, the Heisenberg (XYZ) model with open boundary conditions \cite{huckle-networks-2013,huckle-symmetries-2013}, acting on $\bigotimes\limits_{i=1}^d \mathbb{C}^2$, reads
$$
\begin{array}{rcl}
H & = & j_x H_{xx} + j_y H_{yy} + j_z H_{zz} + \lambda H_x, \\
H_{\mu\nu}(\mathbf{x},\mathbf{x'}) & = & \sum\limits_{i=2}^d \left(\prod\limits_{k=1}^{i-2} E_k \right) \cdot P_{\mu} \cdot P_{\nu} \cdot  \left(\prod\limits_{k=i+1}^{d} E_k \right), \\
H_{\mu}(\mathbf{x},\mathbf{x'}) & = & \sum\limits_{i=1}^d \left(\prod\limits_{k=1}^{i-1} E_k \right) \cdot P_{\mu} \cdot  \left(\prod\limits_{k=i+1}^{d} E_k \right),
\end{array}
$$
where $E_k$ are the identity matrices, $P_{\mu}$ are the Pauli matrices ($\mu=x,y,z$),
$$
P_x = \begin{pmatrix}0 & 1 \\ 1 & 0\end{pmatrix}, \quad P_y = \begin{pmatrix}0 & -\mathrm{i} \\ \mathrm{i} & 0\end{pmatrix}, \quad P_z = \begin{pmatrix}1 & 0 \\ 0 & -1\end{pmatrix},
$$
and $j_{\mu}$, $\lambda$ are scalars.
Lemma \ref{cme:lem:cascadic} and Remark \ref{cme:rem:cascadic_rank>1} can be applied straightforwardly,
giving the following rank estimate.
\begin{lemma}
The Heisenberg (XYZ) Hamiltonian admits an explicit rank-7 TT (or Tucker) representation (note that the TT and QTT formats are indistinguishable here since each Pauli matrix is $2 \times 2$).
\end{lemma}
\begin{proof}
Assembling the factors
$$
F^i_{i-1}  =  \begin{bmatrix} j_x P_x & j_y P_y & j_z P_z \end{bmatrix}, \quad F^i_i = \begin{bmatrix}P_x & P_y & P_z\end{bmatrix}^{\top},
$$
we reduce the problem to that is described by Lemma \ref{cme:lem:cascadic} and Remark \ref{cme:rem:cascadic_rank>1}, with $r_i=3$.
Since $\lambda H_x$ is similar to \eqref{eq:A_destr_lapl}, with $\mathrm{trank}(H_x)=2$, the total estimate writes $(2+3)+2=7$.
\end{proof}
If some of $j_x,j_y,j_z$ are equal to zero, the reduced models appear,
for example, the Heisenberg (XY) Hamiltonian with $j_z=0$ and TT rank 6, or Ising (ZZ) model with $j_x=j_y=0$ and TT rank 5.

\section{Computational scheme} \label{sec:comput}
\subsection{Solving the dynamical problem}\label{sec:dyn_solution}
An exact solution of the CME \eqref{cme:eqn:cme-shift} is written as $P(t) = \exp(At)P(0)$.
However, neither the whole matrix $\exp(At)$ nor even a vector $\exp(At)P(0)$ for a large $t$ is feasible to compute for such a high-dimensional problem.

Instead, the time interval of interest $[0,T]$ may be split into smaller subintervals \\
$[0,T_0],~[T_0,2T_0],\ldots,[T-T_0,T],$ $T_0=T/Q,$
in each of them the initial value problem
$$
\dfrac{dP(t)}{dt} = AP(t), \qquad t \in \left[(q-1)T_0, qT_0\right], \quad q=1,\ldots,Q
$$
is solved using some cheaper scheme.
For example, if $T_0$ is small enough, the implicit Crank-Nicolson time integration may be applied,
\begin{equation}
\left(\mathbf{J}^0-\frac{T_0}{2}A\right)P(qT_0) = \left(\mathbf{J}^0+\frac{T_0}{2}A\right) P((q-1)T_0), \quad q=1,\ldots,Q,
\label{cme:eqn:cranc}
\end{equation}
where $\mathbf{J}^0$ is the identity matrix of the same sizes as $A$ (cf. \eqref{cme:eqn:shift_matrices}).
This scheme is known to be second order accurate, i.e.
$$
\|\exp(A T_0) P((q-1)T_0) - P(q T_0)\| = \mathcal{O}(T_0^3).
$$
However, one issue arising in complex biological system modeling is the long time integration, such that the time scale $T$ is of the order of $100-1000$, and the successive integration using Formula \eqref{cme:eqn:cranc} requires too many steps for a reasonable accuracy (e.g. $10^{-4}$) and corresponding small $T_0$.

Therefore, one needs a balance: $T_0$ is worth to be increased, but the integration scheme on $\left[(q-1)T_0, qT_0\right]$ should be more accurate.
The latter goal may be achieved by further splitting of the range $T_0$ into $N_t$ smaller steps of size $\tau=T_0/N_t$, such that $\mathcal{O}(\tau^2)$ accuracy of the Crank-Nicolson scheme is satisfactory.
Contrarily to the straightforward propagation \eqref{cme:eqn:cranc}, we now aggregate these $\tau$-steps into one global linear system  \cite{DKhOs-parabolic1-2012},
\begin{equation}
\begin{bmatrix}(J_{N_t}^{0}-J_{N_t}^{-1}) \otimes \mathbf{J}^0 - \frac{\tau}{2} (J_{N_t}^{0}+J_{N_t}^{-1}) \otimes A \end{bmatrix} P =
\delta_1 \otimes \begin{pmatrix}P(t_{0})+\frac{\tau}{2} A P(t_{0})\end{pmatrix},
\label{cme:eqn:global_time}
\end{equation}
where $P = \{P(t_n)\}_{n=1}^{N_t} \in \mathbb{R}_+^{(\prod_{i=1}^d N_i)N_t}$ is the vector of all time snapshots at $t_n=(q-1)T_0+\tau n$ stacked,
$J^{-1}_{N_t} \in \mathbb{R}^{N_t \times N_t}$ is the down shift matrix according to \eqref{cme:eqn:shift_matrices}, $J_{N_t}^{0}$ is the identity matrix of size $N_t$, and $\delta_1$ is its first column.
Choosing $N_t=2^L$, this system can be approximated either in the linear QTT or the QTT-Tucker formats, with the quantization in both the state and time variables.
Provided the solution on $\left[(q-1)T_0, qT_0\right]$ is smooth enough, the tensor ranks of $P$ remain moderate, thus the complexity becomes logarithmic w.r.t. the number of steps.
%
Computational benefits of the global scheme w.r.t. the traditional time stepping one have been demonstrated in \cite{DKhOs-parabolic1-2012}.
See also \cite{schwab-spacetime-2009,andreev-bpx-2012} for other versions of the simultaneous space-time discretization approach.


As the linear solver to \eqref{cme:eqn:global_time} in the structured tensor format, we employ
the AMEn algorithm, see Section \ref{sec:amrs}.

\subsection{Computing the steady state}\label{sec:stationary}
In some cases, the transient processes in the CME model are not of interest, and only an approximation to the stationary distribution is required.
In this case, the accurate block formulation \eqref{cme:eqn:global_time} is superfluous.

Instead, we use the implicit Euler iteration,
\begin{equation}
\left(I-T_0 A\right)P(t_{q}) = P(t_{q-1}), \quad q=1,\ldots,Q,
\label{cme:eqn:euler}
\end{equation}
solving the linear system in the left hand side via e.g. Alg. \ref{alg:amen}.
The intermediate solutions approximate the transient processes poorly, but as soon as the method converges, it recovers a component in the lowest eigenspace of $A$ accurately.
As an additional cost reduction, we can use the following trick.
The lowest eigenvalue is of the order of the FSP error, i.e. we may consider it to be zero within the desired accuracy.
So, compute the (normalized) residual as $\eta = \frac{||A P||}{||P||}$.
If $\eta$ is large, we do not need to solve \eqref{cme:eqn:euler} very precisely.
When $\eta$ diminishes, the accuracy may be improved.
Practically, a rule of the form $\eps = c\eta$ (e.g. $c=10^{-1}$), where $\eps$ is the tensor rounding and solution tolerance, is used.
This approach decreases the complexity of intermediate iterations significantly.

In the following we will refer to this method as the (implicit) Euler iterations.
No ambiguity arises since the global state-time formulation \eqref{cme:eqn:global_time} uses the Crank-Nicolson scheme.

As was discussed in Section \ref{sec:amrs}, both the lowest eigenvalue problem, as well as the implicit time scheme solution may be attacked using the traditional Krylov methods,
replacing the standard algebraic operations in the full space by their counterparts in terms of the tensor representations,
but in most cases the TT ranks of the Krylov vectors grow rapidly, and the implicit Euler with the AMEn linear solver overcome this approach.

\section{Numerical experiments}\label{sec:numerics}
The experiments were conducted on a Linux x86\_64 machine with Intel Xeon E5504 processor at 2.00 GHz with the cache size 4096 KB/core.
The Alternating Minimal Energy method (TT and QTT-Tucker versions) for the linear solution and approximation was implemented in MATLAB as a part of the TT-Toolbox\footnote{http://github.com/oseledets/TT-Toolbox} with the most time consuming routines called externally from C/FORTRAN MEX files.

\subsection{20-dimensional signaling cascade}\label{sec:numer_20d}
First, we test the simplest but high-dimensional cascade problem from \cite{hegland-cme-2007,Ammar-cme-2011},
for which we have provided a theoretical analysis of the operator separability.
The model parameters are fixes as follows:
\begin{itemize}
 \item $d=20$, $M=40$;
 \item for $m=1$: $w^{m+}(\mathbf{x}) = 0.7$, $\mathbf{z^{m+}} = -\delta_m$: generation of the first protein;
 \item for $m>1$: $w^{m+}(\mathbf{x}) = \dfrac{x_{m-1}}{5+x_{m-1}}$, $\mathbf{z^{m+}} = -\delta_m$: succeeding creation reactions;
 \item for $m=1,\ldots,20$: $w^{m-}(\mathbf{x}) = 0.07 \cdot x_m$, $\mathbf{z^{m-}} = \delta_m$: destruction reactions.
\end{itemize}
The notations $m+$, $m-$, as well as the operator assembly are according to \eqref{cme:eqn:cme_grads}, and $\delta_m$ is the $m$-th identity vector (in practice one can use the Corollary \ref{cor:neigh+diag} directly).

The computational scheme specifications are the following.
\begin{itemize}
 \item Computational domain $\mathbf{x} \in \{0,\ldots,63\}^{\otimes d}$. The PDF value at $x_i=63$ is below the machine precision.
 \item Linear QTT format for state and time. Both $N$ and linear QTT ranks in this example are small enough, such that the QTT-Tucker format does not bring a further speed-up (which is not the case in the next examples though).
 \item The dynamical problem is solved until $T=400$ via the restarted global state-time solver as proposed in Section \ref{sec:dyn_solution}. We perform an additional test to find an optimal parameter $T_0$.
 \item The initial state is $P(0) = \delta_1 \otimes \cdots \otimes \delta_1$, i.e. all copy numbers are zeros.
 \item Tensor rounding and solution threshold (for the AMEn method) $\eps=10^{-5}$.
\end{itemize}

As a resulting quantity, we compute the average concentrations of all species in time,
$$
\langle x_i \rangle(t) = \dfrac{\sum_{\mathbf{x}} x_i P(\mathbf{x},t)}{\sum_{\mathbf{x}} P(\mathbf{x},t)}, \quad i=1,\ldots,d,
$$
which are presented in Fig. \ref{fig:20d:meanconc}.
One of interesting features in the cascade systems is the delay between the equal concentration levels of different species, which can be observed in Fig. \ref{fig:20d:meanconc}.
Therefore, to keep the time solution history accurately is important to measure such delays.

Additionally, the convergence of the transient solution to the steady state is shown in Fig. \ref{fig:20d:|Au|}, which confirms that the chosen $T$ is enough to approximate the stationary solution with the desired accuracy.

\begin{figure}[h!]
\begin{minipage}[t]{0.49\linewidth}
\centering
\caption{Mean concentrations $\langle x_i \rangle(t)$.}
\label{fig:20d:meanconc}
\includegraphics[width=\linewidth]{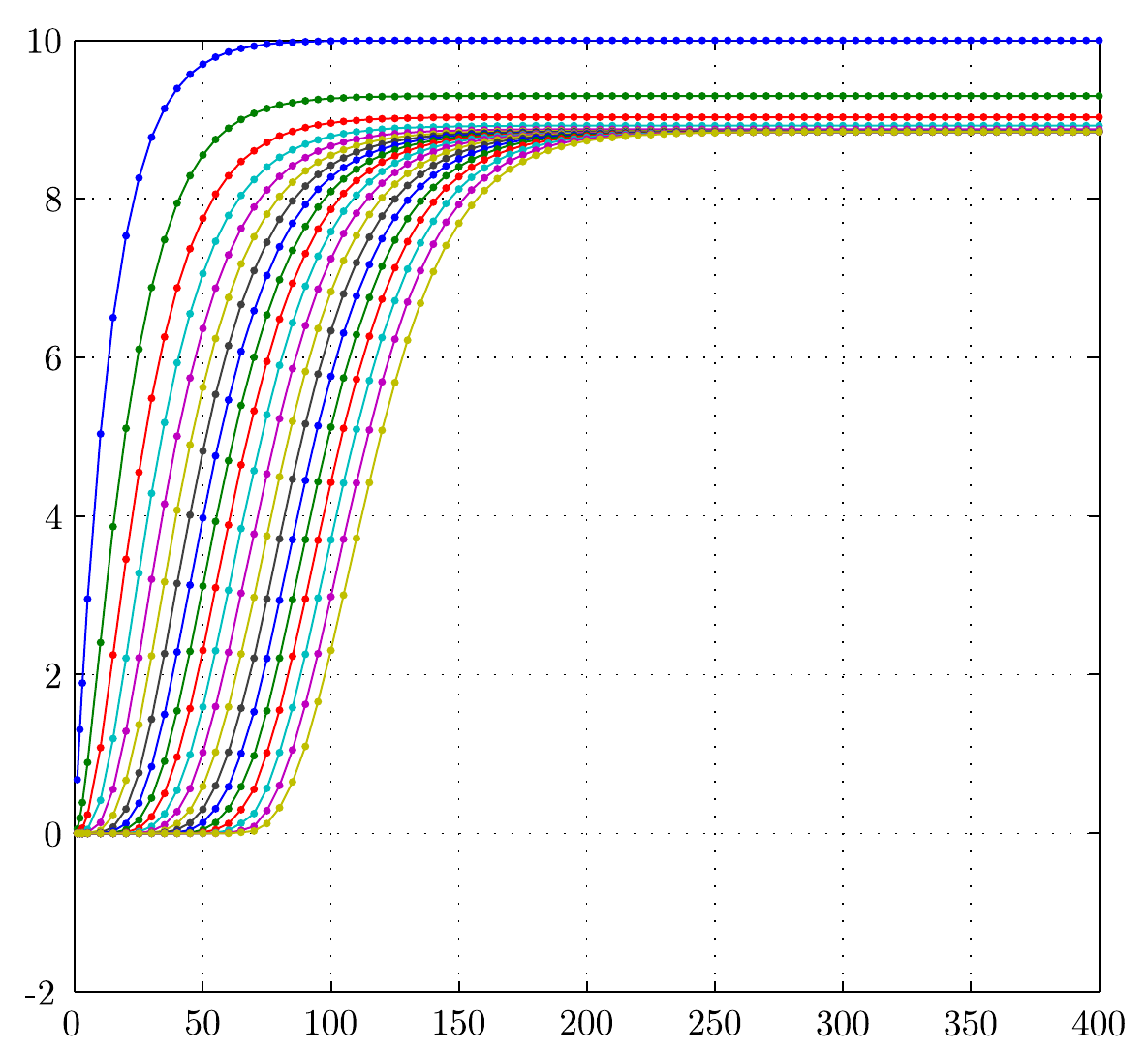}
\end{minipage}
\hfill
\begin{minipage}[t]{0.49\linewidth}
\centering
\caption{Residual $||AP(t)||/||P(t)||$}
\label{fig:20d:|Au|}
\includegraphics[width=\linewidth]{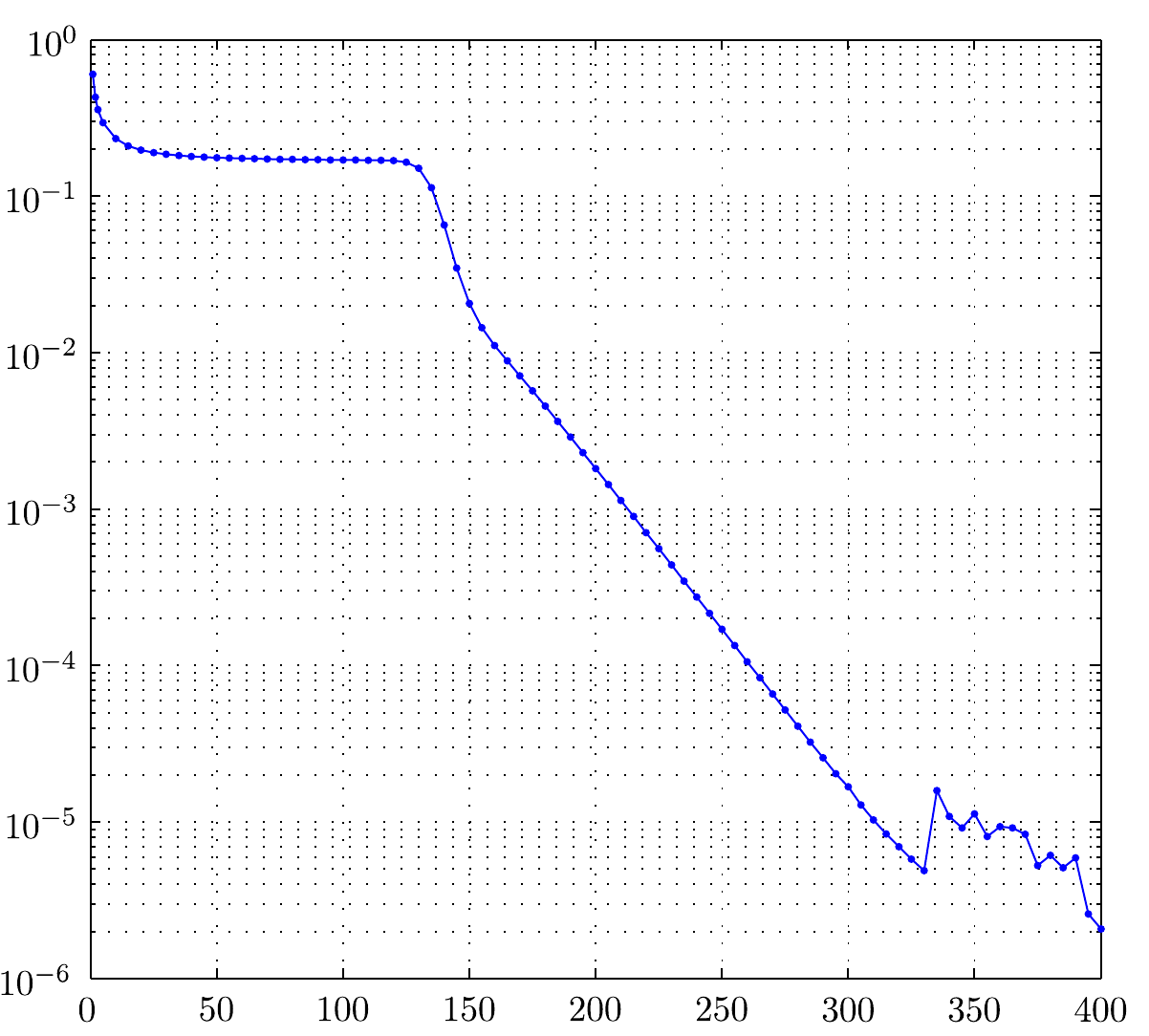}
\end{minipage}
\end{figure}

To demonstrate the performance of the global state-time scheme, we present the CPU times of the solver with different numbers of time steps $N_t$ in each interval $[(q-1)T_0, q T_0]$ (Fig. \ref{fig:20d:ttimes}, left), $q=1,\ldots,T/T_0$, and the time interval widths $T_0$ (Fig. \ref{fig:20d:ttimes}, right).
\begin{figure}[h!]
\caption{CPU time (sec.) vs. discretization parameters. Left: time vs. $\log_2(N_t)$, $T_0=15$. Right: time vs. $T_0$, $N_t=2^{14}$.}
\label{fig:20d:ttimes}
\begin{minipage}[t]{0.49\linewidth}
\centering
\includegraphics[width=\linewidth]{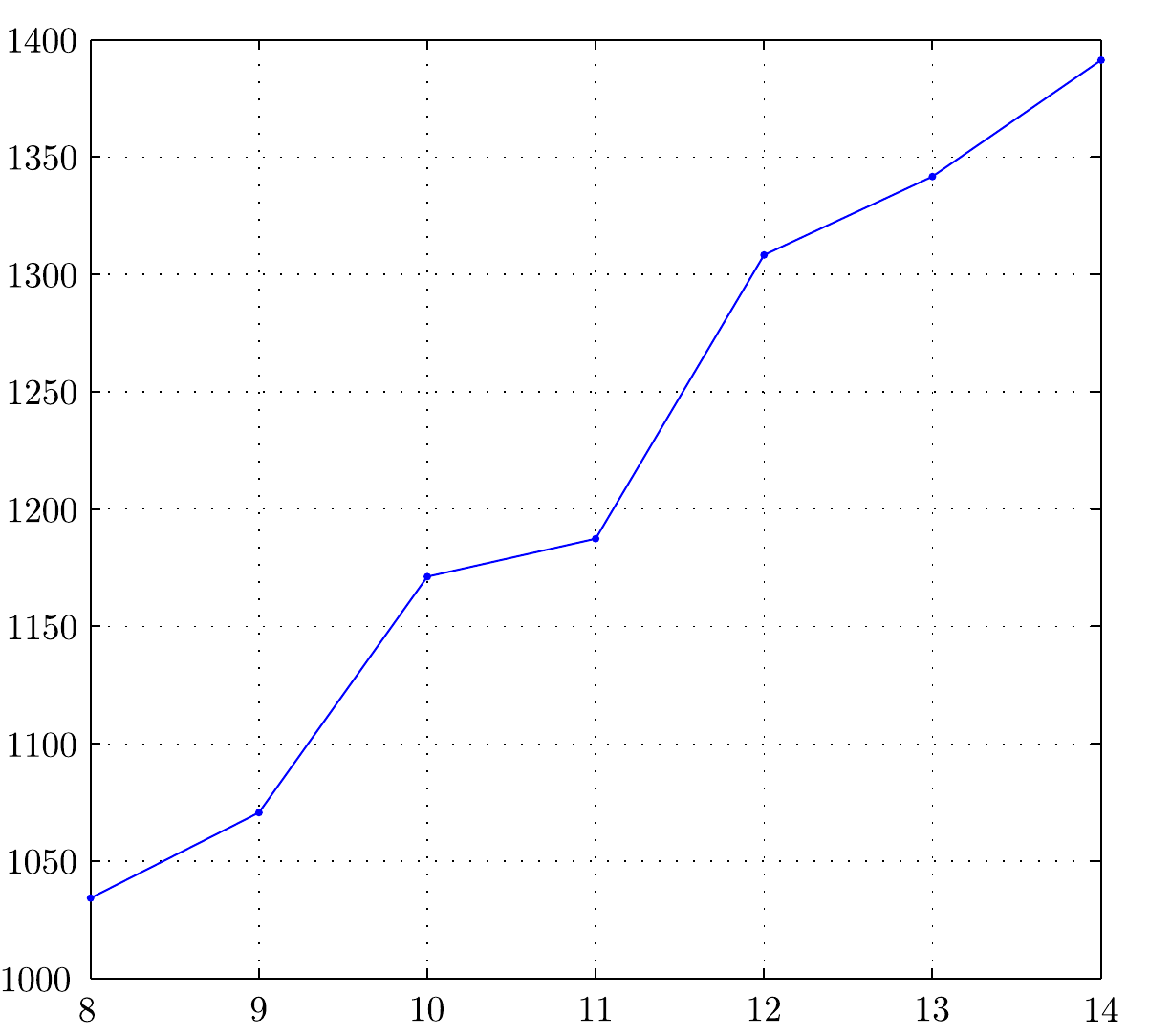}
\end{minipage}
\hfill
\begin{minipage}[t]{0.49\linewidth}
\centering
\includegraphics[width=\linewidth]{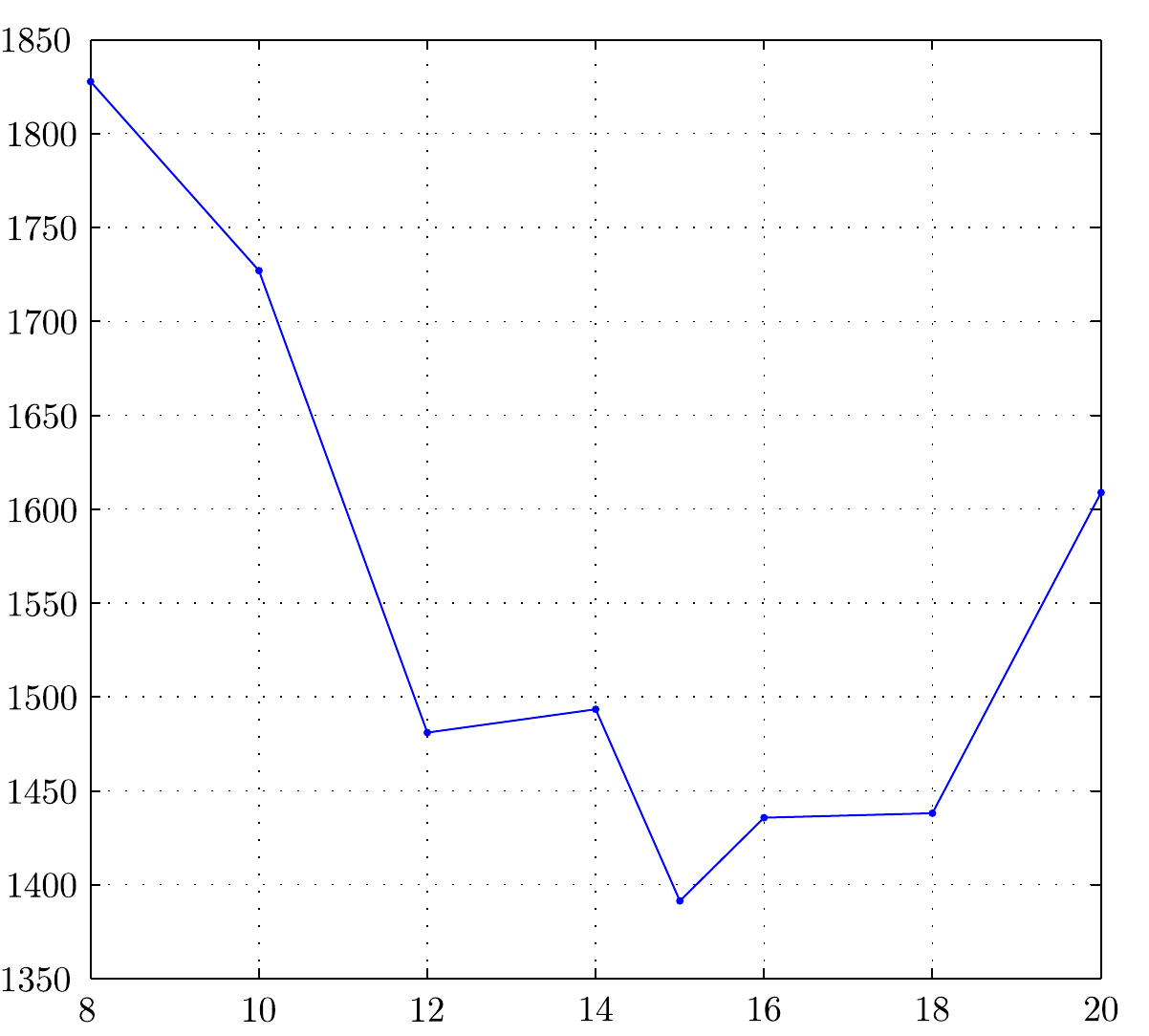}
\end{minipage}
\end{figure}
We see that the computational time grows indeed logarithmically with the time grid size.
The fastest realization was obtained by setting $T_0=15$.
For smaller $T_0$, the solution in each interval is cheap, but the amount of intervals is large.
Contrarily, for large $T_0$ the conditioning (and TT ranks) of each system is high, and it takes more time for the method to converge.

\subsection{A toggle switch in E.Coli with uncertainly defined coefficients}
In this test, we simulate the synthetic bistable genetic toggle switch developed in \emph{Escherichia coli} \cite{gardner-toggle-2000} in presence of a parameter.
The CME model reads
\begin{itemize}
 \item $d=2$, $M=4$:
 \item $w^{1+}(\mathbf{x}) = \dfrac{\alpha_1}{1+x_2^{\beta}}$, $\mathbf{z^{1+}} = -\delta_1$: generation of $S_1$;
 \item $w^{1-}(\mathbf{x}) = x_1$, $\mathbf{z^{1-}} = \delta_1$: destruction of $S_1$;
 \item $w^{2+}(\mathbf{x}) = \dfrac{\alpha_2}{1+\dfrac{x_1}{(1+y/K)^{\eta}}}$, $\mathbf{z^{2+}} = -\delta_2$: generation of $S_2$;
 \item $w^{2-}(\mathbf{x}) = x_2$, $\mathbf{z^{2-}} = \delta_2$: destruction of $S_2$;
 \item $\alpha_1=156.25$, $\alpha_2=15.6$, $\beta=2.5$, $\eta=2.0015$, $K=2.9618 \cdot 10^{-5}$.
\end{itemize}
According to the production rate of $x_1$, we restrict the state space to $\{0,\ldots,255\}^{\otimes 2}$.
A parameter $y$ is the concentration of the IPTG catalyst, and is varying from $10^{-6}$ to $10^{-2}$.
The main feature of this system is two so-called \emph{low} (low concentration of $S_2$) and \emph{high} metastable states.
The probability to find the system in either of states depends on the concentration $y$, see Figure \ref{fig:toggle:meanconc}.

Note that $y$ is not an active specie copy number, but only enters as a parameter in the coefficients, $w^m = w^m(\mathbf{x},y)$, so that
$$
\dfrac{dP(y,t)}{dt} = A(y)  P(y,t) = \sum\limits_{m=1}^M (\mathbf{J}^{\mathbf{z^m}} - \mathbf{J}^0) \diag(w^m(y)) P(y,t) \quad \mbox{for each }y.
$$
Introducing a discretization in the parameter (in this example, the exp-uniform grid is used),
we come to the block-diagonal form of the global system,
$$
\frac{\partial P(t)}{\partial t} = \mathcal{A} P(t) =
\begin{bmatrix}
A(y_1) \\
& \ddots \\
& & A(y_{N_y})
\end{bmatrix}
P(t), \qquad
P(t)=
\begin{bmatrix}
P(y_1,t) \\
\vdots \\
P(y_{N_y},t)
\end{bmatrix},
$$
where $\{y_1,\ldots,y_{N_y}\}$ is the parametric grid, and $A(y_j)$ stands for the original CME matrix \eqref{cme:eqn:cme-shift} with a fixed $y=y_j$, $j=1,\ldots,N_y$.

If the total number of parametric points $N_y$ is moderate,
the very straightforward approach is to solve each CME equation for each $y_j$ independently.
However, if $N_y$ is large ($y$ can represent in fact a \emph{multi}-parameter tuple with $d_y$ variables, resulting in $N_y=N^{d_y}$ degrees of freedom in total),
one may benefit from the tensor-structured data compression,
and solve the global system at once, disregarding its block-diagonality.
The cases of many parameters arise naturally in stochastic equations (e.g. \cite{matthies-galerkin-2005,litvinenko-spde-2013,KhSch-Galerkin-SPDE-2011}), when some coefficients may not be known exactly in advance, but only their possible ranges can be specified, or inverse problems, such as model calibration and sensitivity \cite{sreenath-pathways-2008,rabitz-sensitivity-1983}.
The current example may be assigned to the latter class, but it is not so important from the computational point of view.

Though the tensor structuring is especially efficient for many dimensions,
here we compare the solution of the 3-dimensional $(x_1,x_2,y)$ problem at once in the QTT and QTT-Tucker formats with the standard full-format scheme, which solves the linear systems with sparse matrices in ($x_1,x_2$) independently at all $y$ points,
and show that they may be competitive already for moderate problem sizes.

In our modeling, we seek for the steady state using the Euler iterations (see Section \ref{sec:stationary}) in the time range $[0,1000]$, so that the stationarity accuracy is below the tensor rounding tolerance $\eps=10^{-5}$.
The time step $T_0$ is varied from 1 to 5.
%

First of all, check the computational times versus the parametric grid size and the time step $T_0$ (Figure \ref{fig:toggle:ttimes}).
\begin{figure}[h!]
\caption{CPU time vs. $\log_2(N_y)$. Left: $T_0=1$, middle: $T_0=2$, right: $T_0=5$.}
\label{fig:toggle:ttimes}
\begin{minipage}[t]{0.32\linewidth}
\centering
\includegraphics[width=\linewidth]{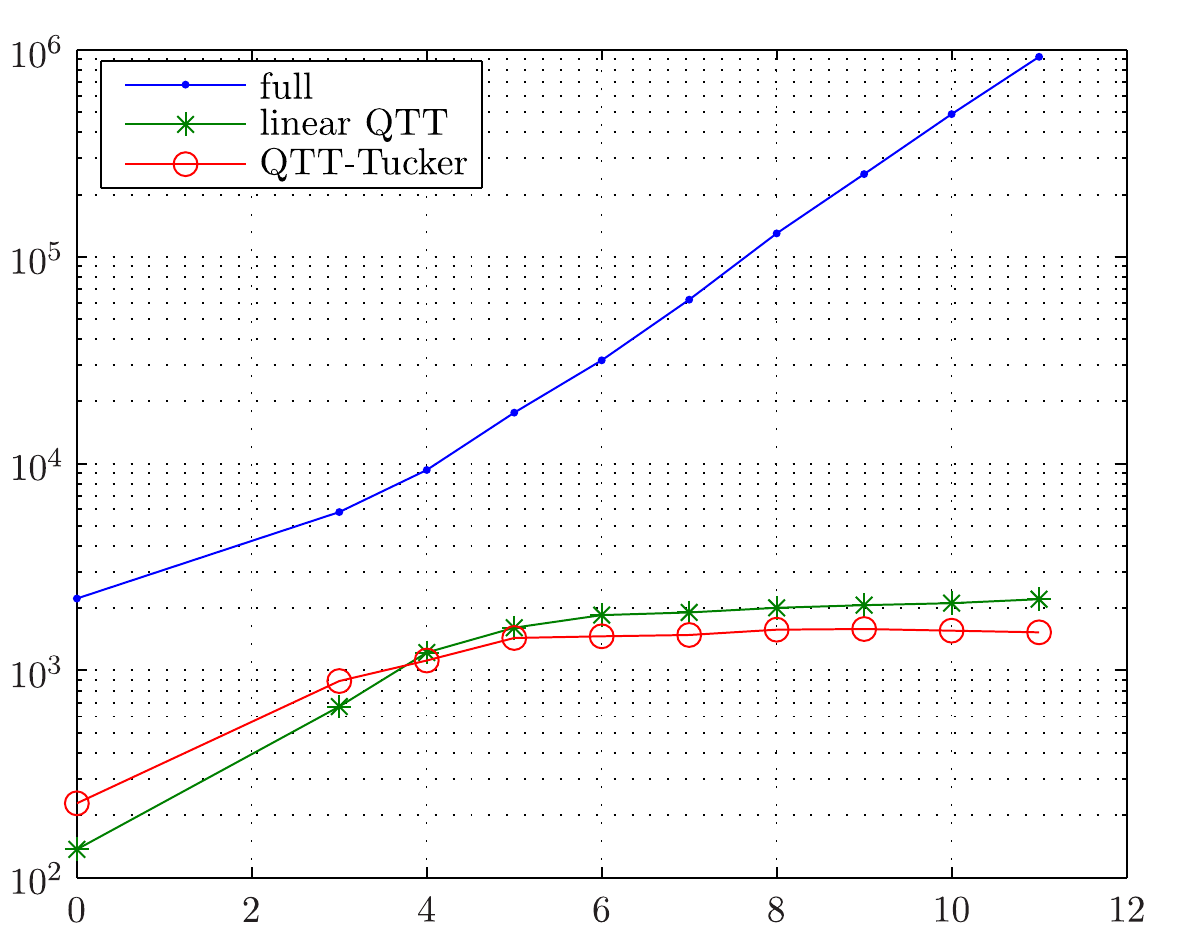}
\end{minipage}
\hfill
\begin{minipage}[t]{0.32\linewidth}
\centering
\includegraphics[width=\linewidth]{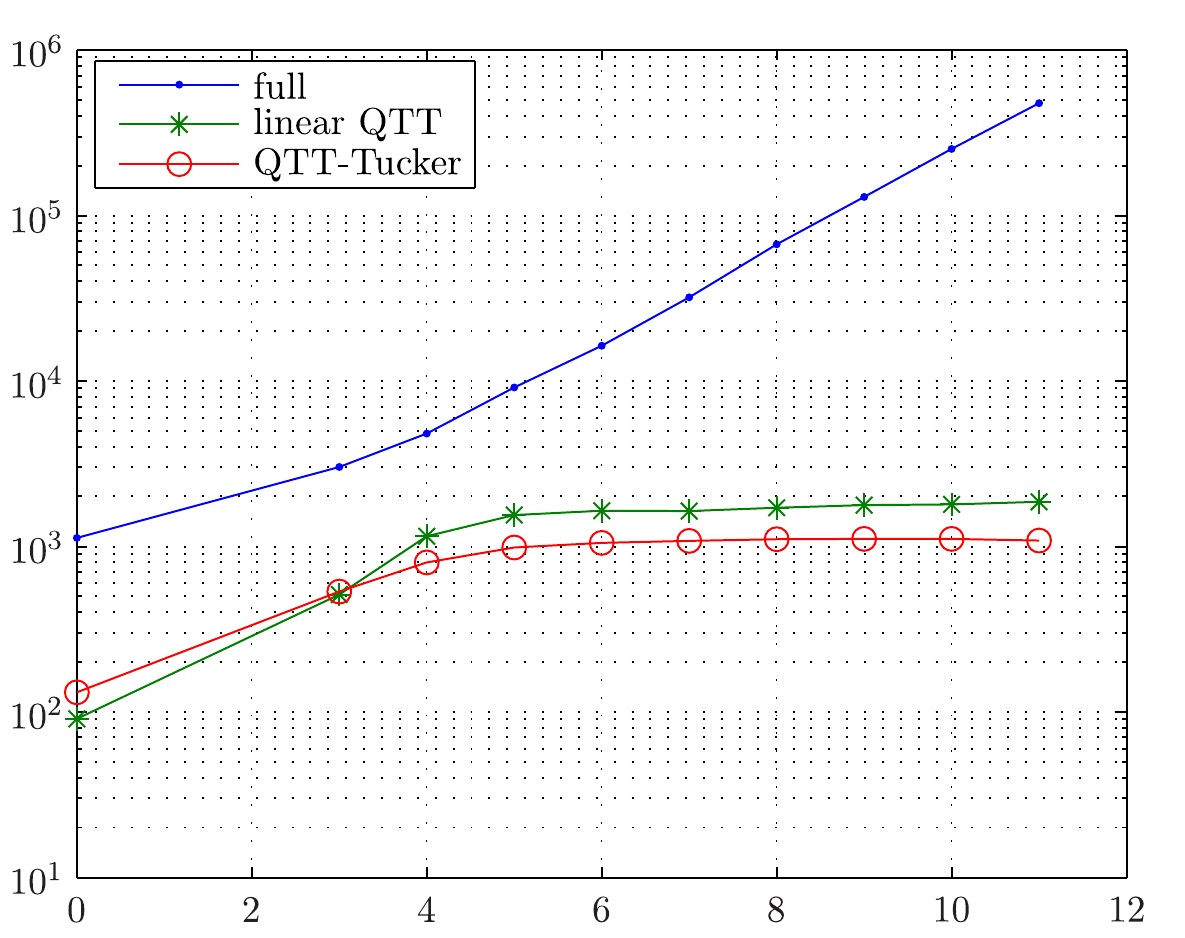}
\end{minipage}
\hfill
\begin{minipage}[t]{0.32\linewidth}
\centering
\includegraphics[width=\linewidth]{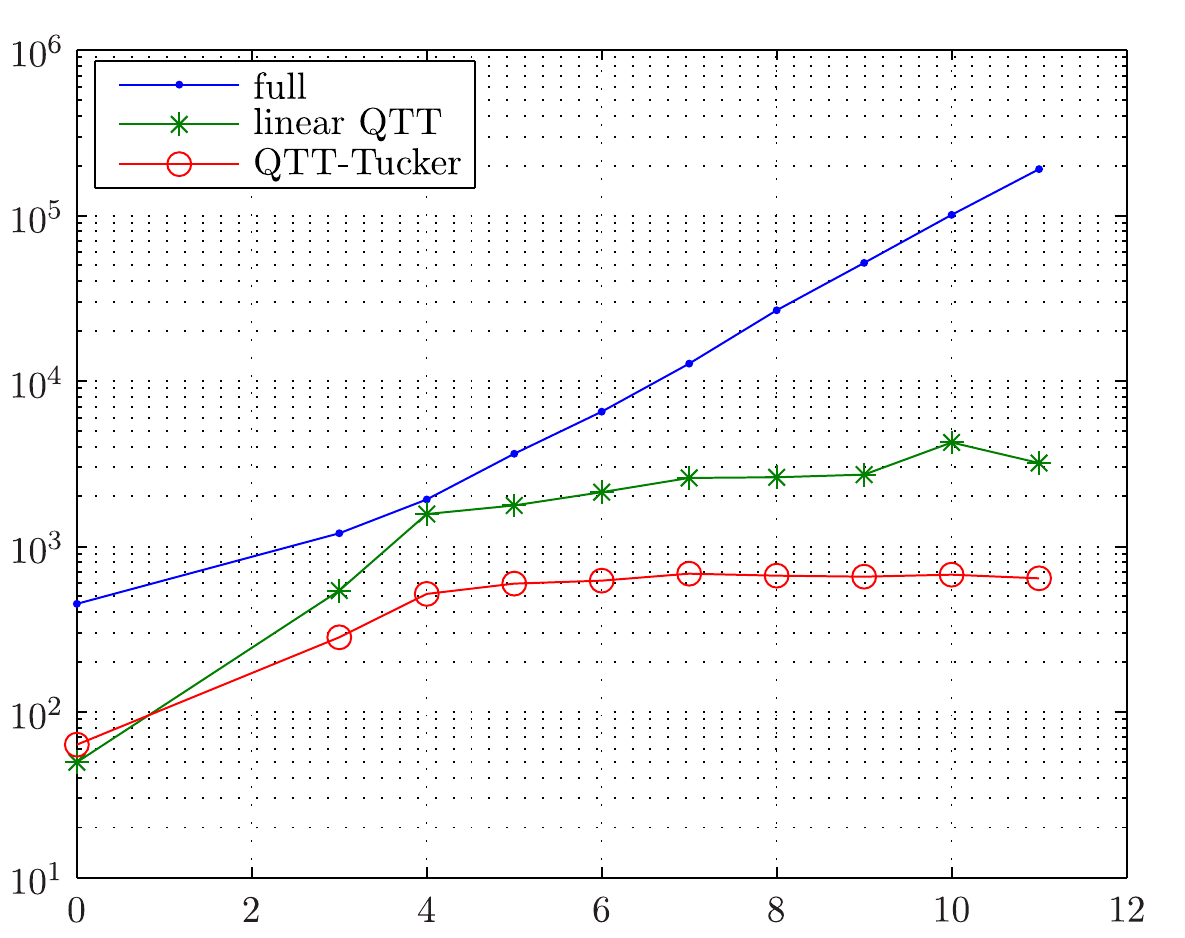}
\end{minipage}
\end{figure}
As expected, the complexity of the ``full'' scheme (independent solution for all parametric points) demonstrates a linear growth w.r.t. $N_y$.
For both QTT-based tensor formats the growth is logarithmic.
Remarkable, even with one parametric point, the tensor schemes overcome the full-format one.
This shows a high efficiency of the proposed methods: with the tensor ranks of the solution of the order of $30-40$, even a two-dimensional problem with $256^2$ grid points can be accelerated.

When $T_0$ and $N_y$ increase, so do the conditioning of the matrix and solution ranks.
In this case, the QTT-Tucker format provides additional complexity reduction, by a factor of ca. 4 w.r.t. the linear QTT.

Since we introduce a tensor rounding error of magnitude $10^{-5}$, we need to check the actual error in the observable quantities, especially when solutions at different parametric points (hence different scales) are approximated in one tensor structure.
We track the point $y=3\cdot 10^{-5}$ (it is located in the transient region, see Fig. \ref{fig:toggle:meanconc}) by adding it explicitly to the grid, and check the error in the average concentrations, $\frac{|\langle x_i \rangle - \langle x_i \rangle_{ex}|}{\langle x_i \rangle_{ex}}$,
where ${\langle x_i \rangle_{ex}}$ is taken from the simulation in the full format, see Table \ref{tab:toggle}.
For brevity, we show only the mean and maximal errors over different $N_y$ and $T_0$.
We see that the error may increase by a factor up to 250 w.r.t. the solution tolerance,
but usually an accuracy $\mathcal{O}(10^{-3})$ is enough for such phenomenological models, while the complexity may be substantially reduced.

\begin{minipage}[t]{0.45\linewidth}
\centering
\captionof{table}{Mean and maximal relative errors of $\langle x_i \rangle$ at $y=3\cdot 10^{-5}$}
\label{tab:toggle}
\begin{tabular}{c|cc|cc}
            & \multicolumn{2}{c|}{linear QTT} & \multicolumn{2}{c}{QTT-Tucker} \\ \hline
            & $x_1$   & $x_2$    & $x_1$   & $x_2$    \\ \hline
mean        & 8.8e-4 &  7.8e-5 & 5.5e-4 & 4.9e-5  \\
max         & 2.5e-3 &  2.2e-4 & 2.1e-3 & 1.8e-4 \\
$T_0^{\max}$ & 1      &  1      &   1    &  1    \\
$N_y^{\max}$  & $2^7$  &  $2^7$  &  $2^6$ & $2^6$ \\
\end{tabular}
\end{minipage}
\hfil
\begin{minipage}[t]{0.45\linewidth}
 \centering
 \captionof{figure}{Mean concentrations $\langle x_1\rangle$, $\langle x_2\rangle$ versus $y$}
 \label{fig:toggle:meanconc}
 \includegraphics[width=\linewidth]{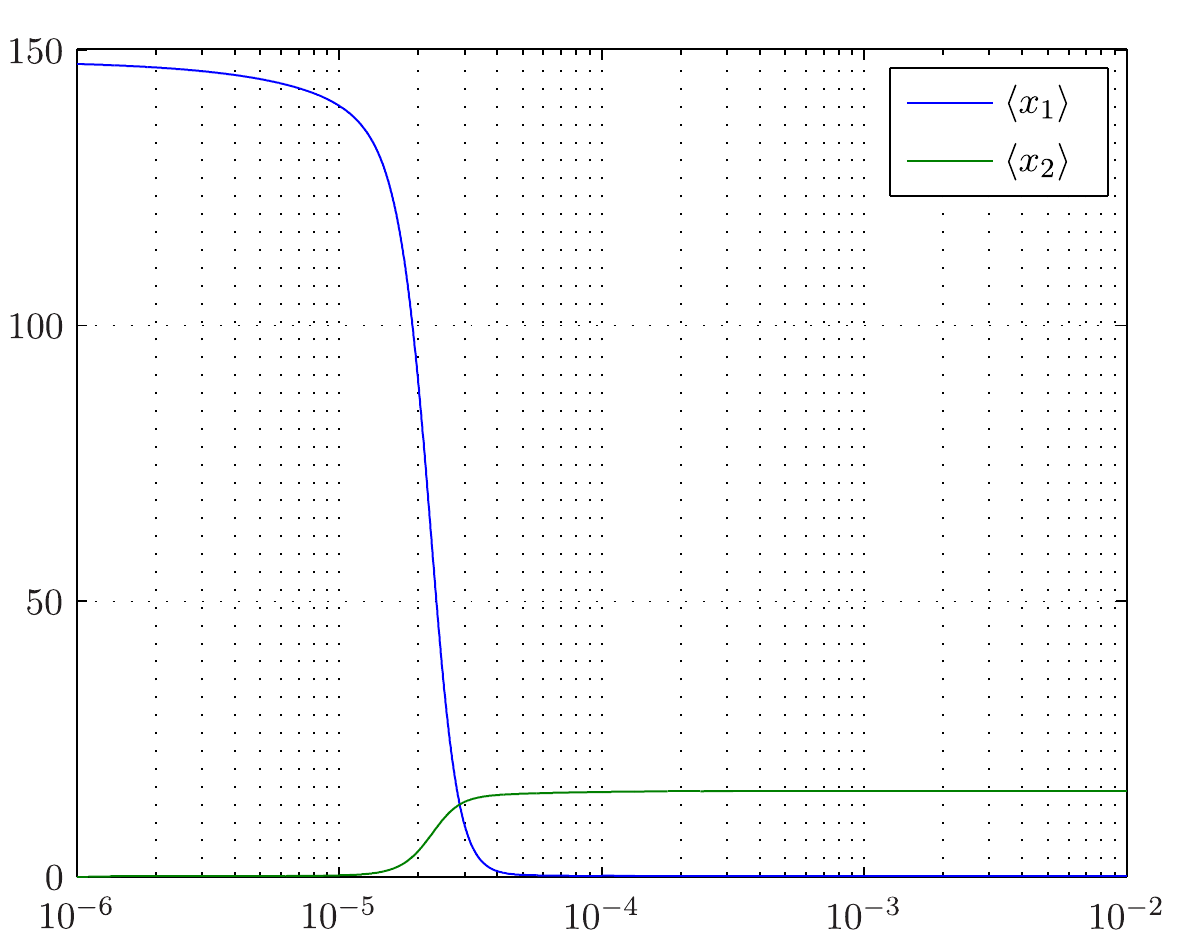}
\end{minipage}\\


Finally, we plot the average concentrations of both reacting proteins versus the concentration of the catalyst, see Figure \ref{fig:toggle:meanconc}.
The quantity $\langle x_2\rangle$ can be used also as a measure of the fraction of the cells occupying the high (low) state.
Indeed, it demonstrates the asymptotic: with $y$ tending to $\infty$, the cells tend to stay in the high state.
As a result, normalizing $\langle x_2\rangle$ to its maximal value, obtain the fraction of the high-state cells.
It is the plot that was given in \cite{gardner-toggle-2000}, Fig. 5(b), and we observe a good agreement with the experimental results.

\subsection{$\lambda$-phage}\label{sec:lphage-num}
The last example is the simulation of the life cycle of the bacteriophage-$\lambda$ \cite{hegland-cme-2007,jahnke-cme-2008}.
The first paper \cite{hegland-cme-2007} considers the sparse grids approach.
The second one is more tensor related and uses the so-called \emph{Dirac-Frenkel} principle for the dynamical low-rank approximation in the Tucker format (DLRA).
However, the simulation was done for a small time interval ($T=10$) only, which is far from providing the stationary solution.
The reason is that the stationary concentration of the second specie is too high ($\sim 10^4$) to be efficiently treated by the algorithms proposed.
In this article, we present an efficient computation of the stationary solution on very large grids with the use of the QTT format and the alternating linear solver.

The model parameters read
\begin{itemize}
 \item $d=5$, $M=10$.
 \item $w^{1+}(\mathbf{x}) = \dfrac{a_1 b_1}{b_1+x_2}$, $\mathbf{z^{1+}} = -\delta_1$: generation of $S_1$; \quad $a_1=0.5$, $b_1=0.12$.
 \item $w^{1-}(\mathbf{x}) = c_1 \cdot x_1$, $\mathbf{z^{1-}} = \delta_1$: destruction of $S_1$; \quad $c_1=0.0025$.
 \item $w^{2+}(\mathbf{x}) = \dfrac{(a_2+x_5)b_2}{b_2 + x_1}$, $\mathbf{z^{2+}} = -\delta_2$: generation of $S_2$; \quad $a_2=1$, $b_2=0.6$.
 \item $w^{2-}(\mathbf{x}) = c_2 \cdot x_2$, $\mathbf{z^{2-}} = \delta_2$: destruction of $S_2$; \quad $c_2=0.0007$.
 \item $w^{3+}(\mathbf{x}) = \dfrac{a_3 b_3 x_2}{b_3\cdot x_2+1}$, $\mathbf{z^{3+}} = -\delta_3$: generation of $S_3$; \quad $a_3=0.15$, $b_3=1$.
 \item $w^{3-}(\mathbf{x}) = c_3 \cdot x_3$, $\mathbf{z^{3-}} = \delta_3$: destruction of $S_3$; \quad $c_3 = 0.0231$.
 \item $w^{4+}(\mathbf{x}) = \dfrac{a_4 b_4 x_3}{b_4\cdot x_3+1}$, $\mathbf{z^{4+}} = -\delta_4$: generation of $S_4$; \quad $a_4=0.3$, $b_4=1$.
 \item $w^{4-}(\mathbf{x}) = c_4 \cdot x_4$, $\mathbf{z^{4-}} = \delta_4$: destruction of $S_4$; \quad $c_4 = 0.01$.
 \item $w^{5+}(\mathbf{x}) = \dfrac{a_5 b_5 x_3}{b_5\cdot x_3+1}$, $\mathbf{z^{5+}} = -\delta_5$: generation of $S_5$; \quad $a_5=0.3$, $b_5=1$.
 \item $w^{5-}(\mathbf{x}) = c_5 \cdot x_5$, $\mathbf{z^{5-}} = \delta_5$: destruction of $S_5$; \quad $c_5=0.01$.
\end{itemize}
The matrix assembly is done by summing the rank-2 Laplace-like TT decomposition of the destruction part and 5 creation parts, obtaining the total TT rank bound $7$.

First of all, we give a comparison with the dynamical Tucker approximation algorithm from \cite{jahnke-cme-2008}.
That paper reports that the grid of sizes $15 \times 40 \times 10 \times 10 \times 10$ was used; however, to employ the QTT format, we restrict ourselves to the powers of two, $16 \times 64 \times 16 \times 16 \times 16$.
The time interval is $T=10$, and we employ the accurate time integration to resolve the transient processes,
by solving the coupled state-time system \eqref{cme:eqn:global_time} for different inner time intervals $T_0$ and time grid sizes $N_t$.
As the initial state, the ($n=3$, $p=[0.05,...,0.05]$)-multinomial distribution was chosen,
$$
P(\mathbf{x},0) =
\dfrac{3!}{x_1! \cdots x_5! \cdot (3-|x|)!} 0.05^{|x|} (1-5 \cdot 0.05)^{3-|x|} \cdot \theta(3-|x|),
$$
where $|x| = x_1 + \cdots + x_5$, and $\theta(\xi)$ is the Heaviside function.
This function can be constructed straightforwardly as a full-format $4 \times 4 \times 4 \times 4 \times 4$-tensor thanks to the zeroing Heaviside function if any of $x_i$ is greater than $3$.
After that, the TT decomposition (with ranks 4) is computed, and each factor is expanded by zeros to the appropriate grid size.
In the end, the TT representation is reapproximated into the QTT one.

The CPU timings are presented in Table \ref{tab:lphage:ttimes}.
\begin{table}[h!]
\centering
\caption{CPU times (sec.) versus $T_0$ and $N_t$, small time range}
\label{tab:lphage:ttimes}
\begin{tabular}{c|cccc||c||c}
 \multicolumn{5}{c||}{linear QTT}                           & Full    & DLRA \cite{jahnke-cme-2008} \\ \hline
 $T_0$ $\backslash$ $N_t$ & 256 & 512 & 1024 & 2048        & $T_0=0.01$   &  \\ \hline
 1 & 20.66 &  23.89 &  21.18 &  20.38                      & 1071            & 300 \\
 2 & 18.49 &  19.52 &  19.44 &  17.70 \\
 5 & 16.33 &  16.35  & 17.55  & 16.01 \\
 10 & 25.34  & 15.51 &  12.49 &  11.23 \\
\end{tabular}
\end{table}
With such small QTT ranks ($\sim 30$), the CPU times are almost independent on the time grid size, and even tend to decrease, since the finer discretization provides more accurate and smooth solution.
The same situation may be observed with respect to $T_0$ as well.
In all cases the computation is much faster than 5 minutes of the DLRA, and a fortiori than $\sim 3$ hours of the SSA, reported in \cite{jahnke-cme-2008}.

To check the accuracy, we compare the marginal probability densities with those are computed in the full format on the grid $16 \times 64 \times 10 \times 10 \times 10$.
The corresponding Crank-Nicolson propagation matrices of size $1024000$ were assembled in the MATLAB \texttt{sparse} format, with the use of \texttt{gmres} as the iterative solver.
The integration was conducted with the time step $0.01$, which required $1071$ seconds of CPU time (see Table \ref{tab:lphage:ttimes}).
The marginal distributions are shown in Figure \ref{fig:lphage:Pi}, and the 2-norm errors are presented in Table~\ref{tab:lphage:errs}.
\begin{figure}[h!]
\centering
\caption{Left to right, top: $P_1(x_1)$, $P_2(x_2)$, $P_3(x_3)$. Bottom: $P_4(x_4)$, $P_5(x_5)$}
\label{fig:lphage:Pi}
\begin{minipage}[t]{0.30\linewidth}
\centering
\includegraphics[width=\linewidth]{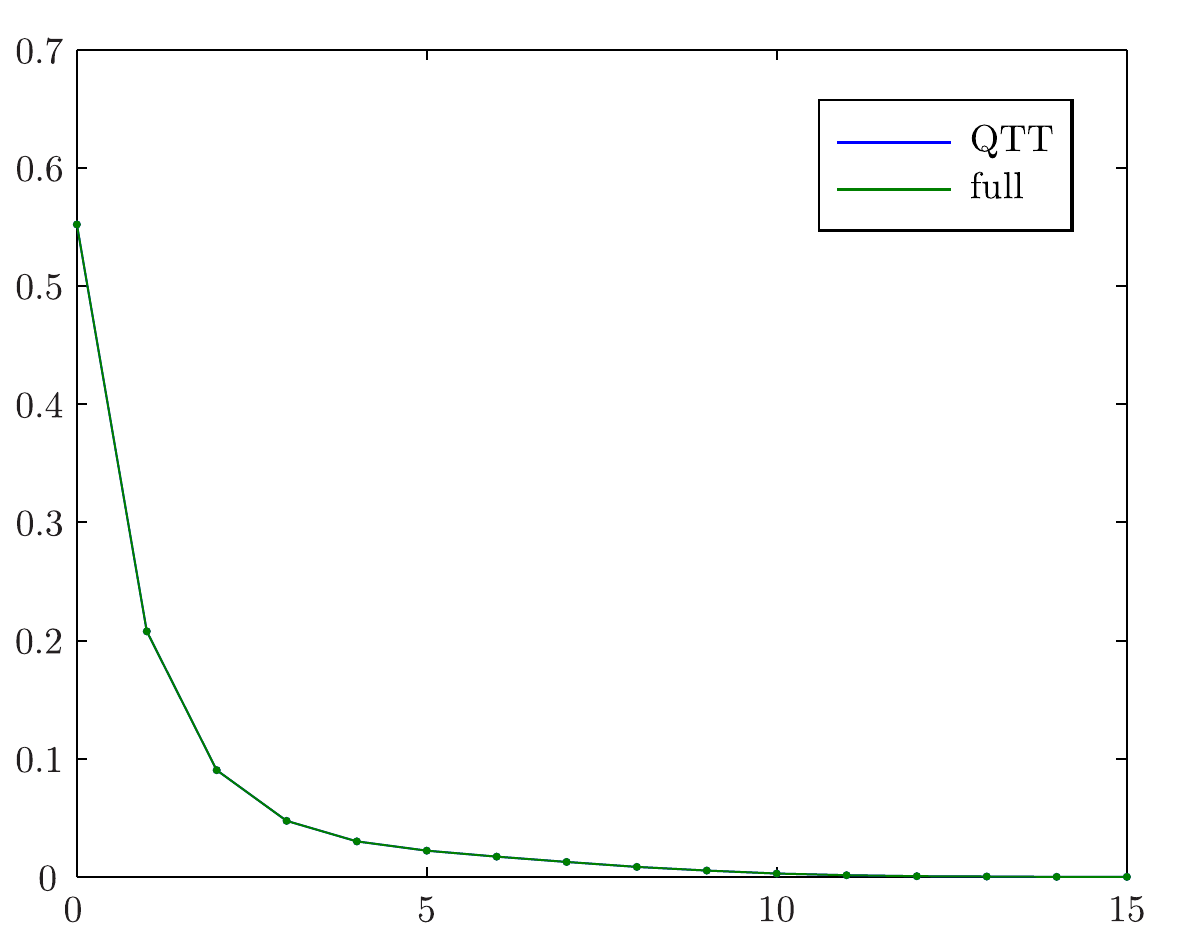}
\end{minipage}
\hfill
\begin{minipage}[t]{0.30\linewidth}
\centering
\includegraphics[width=\linewidth]{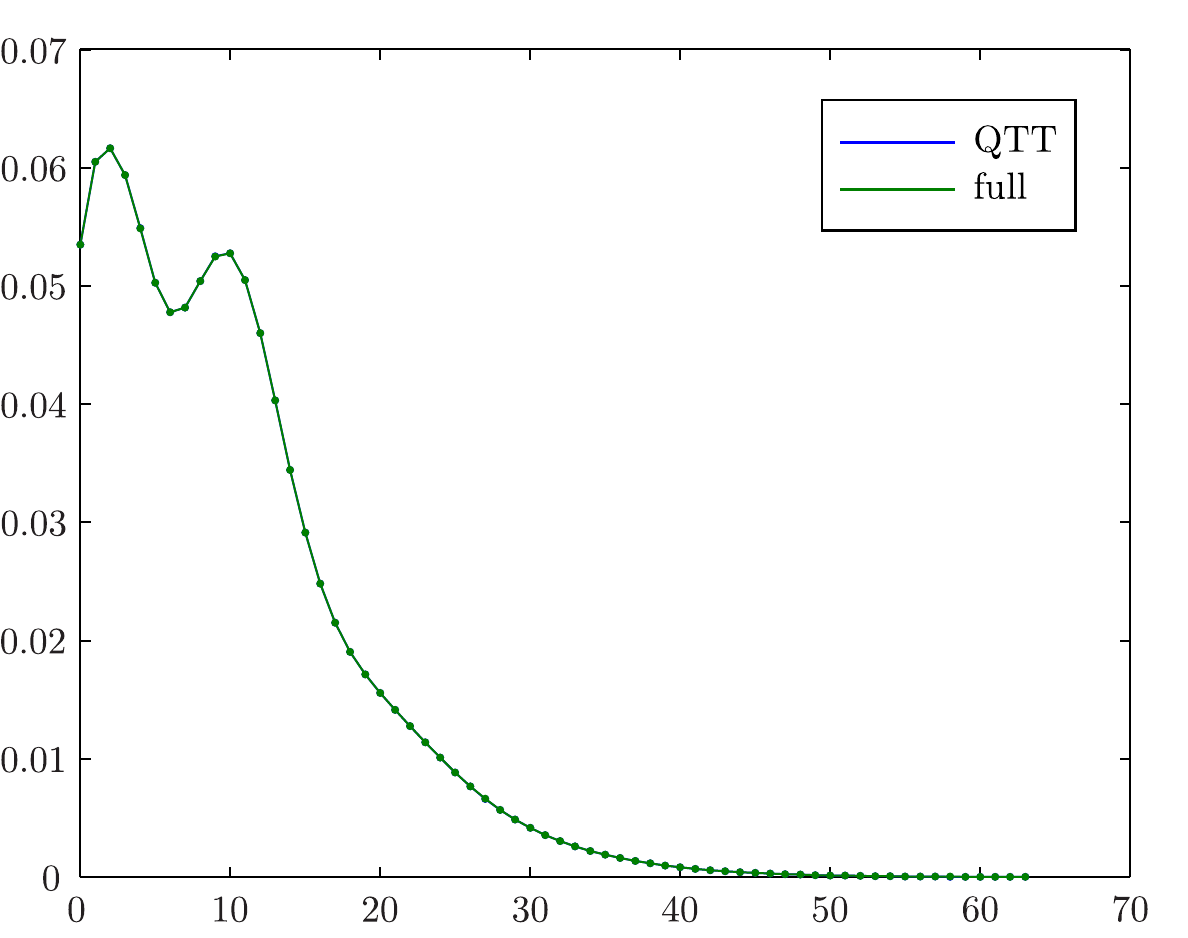}
\end{minipage}
\hfill
\begin{minipage}[t]{0.30\linewidth}
\centering
\includegraphics[width=\linewidth]{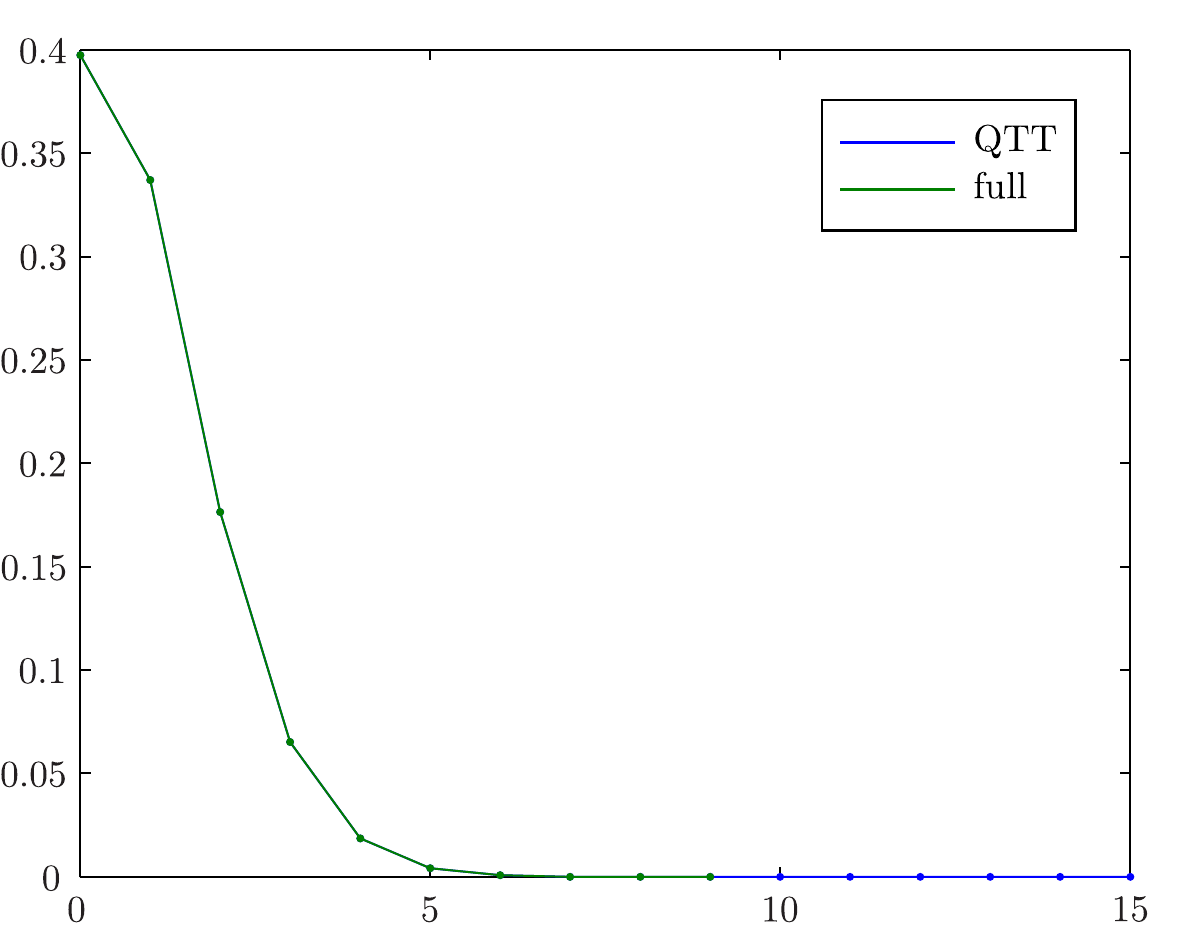}
\end{minipage}
\\
\begin{minipage}[t]{0.30\linewidth}
\centering
\includegraphics[width=\linewidth]{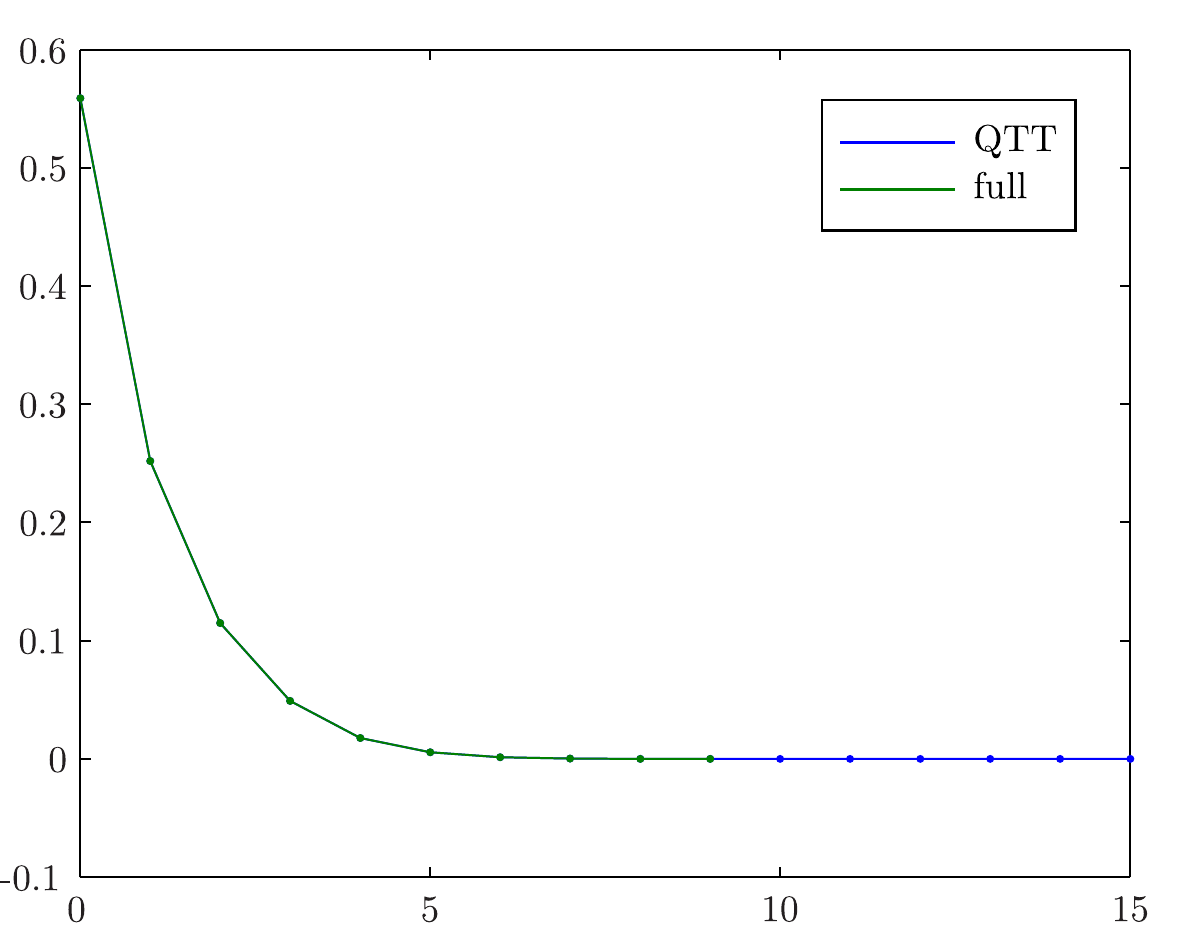}
\end{minipage}
\begin{minipage}[t]{0.15\linewidth}
\hspace*{\linewidth}
\end{minipage}
\begin{minipage}[t]{0.30\linewidth}
\centering
\includegraphics[width=\linewidth]{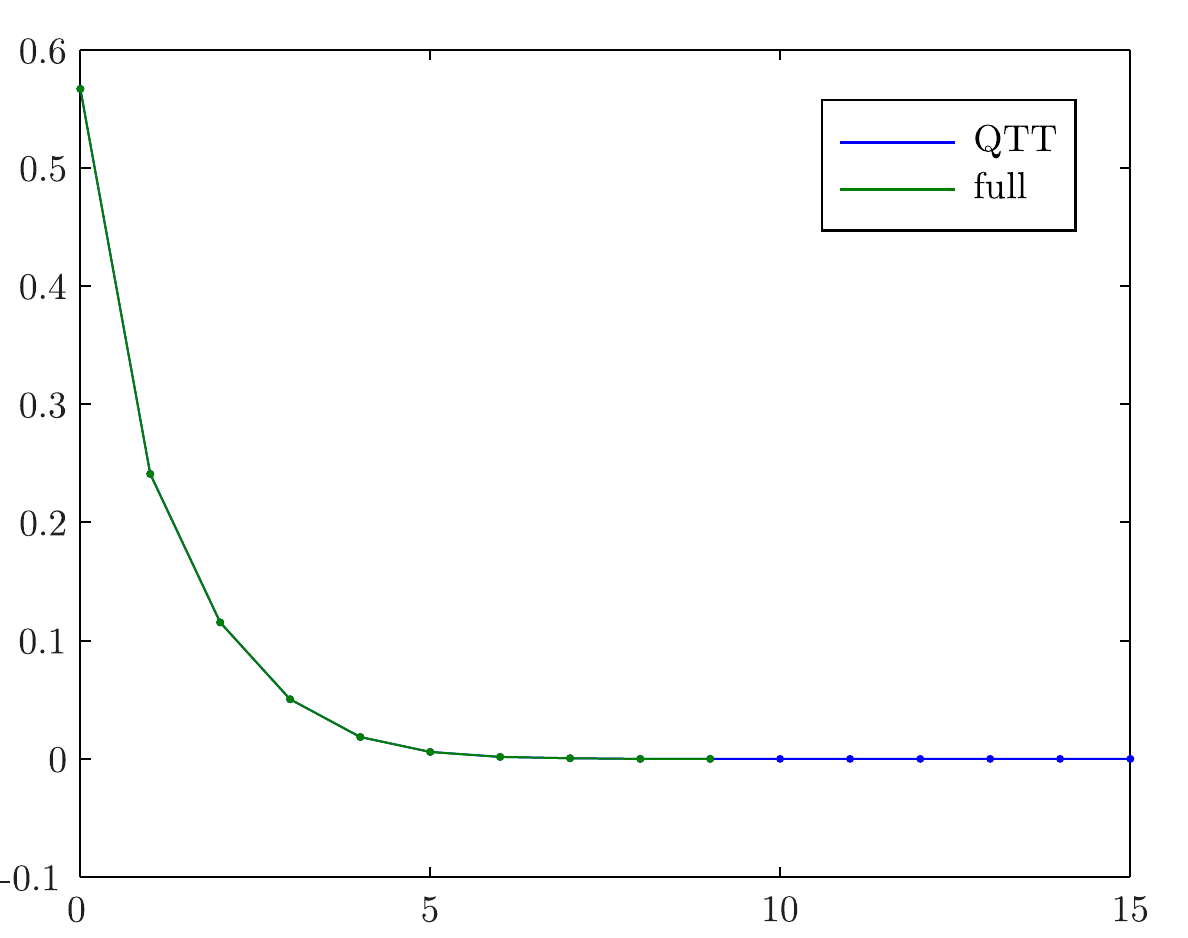}
\end{minipage}
\end{figure}
\begin{table}[h!]
\centering
\caption{Accuracy versus $T_0$ and $N_t$, small time range}
\label{tab:lphage:errs}
\begin{tabular}{c|cccc||c}
\multicolumn{5}{c||}{$\frac{||P_{qtt}-P_{full}||}{||P_{full}||}$}      & $\frac{||P_{DLRA}-P_{SSA}||}{||P_{SSA}||}$ \cite{jahnke-cme-2008} \\ \hline
 $T_0$ $\backslash$ $N_t$ & 256 & 512 & 1024 & 2048             &  \\ \hline
1 & 7.971e-5 & 7.096e-5 & 7.277e-5 & 8.543e-5 			& 3e-3 \\
2 & 4.780e-5 & 6.700e-5 & 4.356e-5 & 7.643e-5  \\
5 & 6.387e-5 & 8.733e-5 & 2.102e-4 & 2.769e-4  \\
10 & 1.435e-4 & 2.296e-4 & 2.933e-4 & 3.963e-4  \\
\end{tabular}
\end{table}
One should note that the time step $0.01$ used in the full-format simulations, as well as the FSP truncation at lower grid sizes yield the error in the full solution $\mathcal{O}(10^{-4})$.
That is, the QTT solution with $N_t=256$, $T_0=2$ ($\tau \sim 0.01$) appears to be closer to the full-format one, than with the finer discretizations.
However, in all cases the accuracy $\mathcal{O}(10^{-4})$ is achieved (note that the solutions in Fig. \ref{fig:lphage:Pi} are indistinguishable), while the tensor rounding tolerance was set to $\eps=10^{-5}$.

To integrate the system until the stationary solution is a much more difficult problem.
First, we set the grid sizes to $128 \times 65536 \times 64 \times 64 \times 64$, in accordance to very high ($\sim 4 \cdot 10^4$) concentrations of the second specie, see Fig. \ref{fig:lphage:meanconc}.
Second, the relaxation time is large, $T \sim 2\cdot 10^4$.
In our simulation, we use the exponential splitting of the time interval,
$$
t_q = \exp(0.05 \cdot q), \quad q=1,...,200.
$$
To obtain an accurate time history, in each subinterval $[t_{q-1}, t_q]$ we solve the coupled state-time system with an additional splitting into $1024$ time steps (encapsulated in the QTT format).
The convergence history and the cumulative CPU times are shown in Figure \ref{fig:lphage:Au-ttimes} and Table \ref{tab:lphage:ttimes_large}.
We see that it takes about an hour of computational time to recover the whole time history with the stationarity accuracy $\sim 10^{-7}$.
Despite the large grids, this is the case, when the QTT-Tucker does not outperform the linear QTT format due to large core ranks (i.e. the physical dimensions are strongly connected), and its larger (cubic) asymptotic leads to a larger time. In this example, the QTT-Tucker solver takes about $4000$ seconds.

\begin{table}[h!]
\begin{minipage}[t]{0.48\linewidth}
\centering
\caption{CPU times (sec.), large time range}
\label{tab:lphage:ttimes_large}
\begin{tabular}{cc||cc}
\multicolumn{2}{c||}{linear QTT} & \multicolumn{2}{c}{QTT-Tucker} \\ \hline
$N_t=2^{10}$ & $N_t=1$           & $N_t=2^{10}$ & $N_t=1$ \\ \hline
3500         & 670               & 4000         & 410
\end{tabular}
\end{minipage}
\hfil
\begin{minipage}[t]{0.48\linewidth}
\centering
\caption{Accuracies of $\langle x_i \rangle(T)$.}
\label{tab:lphage:steady_acc}
\begin{tabular}{ccccc}
$S_1$ & $S_2$ & $S_3$ & $S_4$ & $S_5$ \\
4.6e-5 &   1.7e-6  &  6.2e-8  &  3.1e-7 &  1.7e-7. \\
\end{tabular}
\end{minipage}
\end{table}

If we are not interested in the transient processes, we may use the implicit Euler iterations over the time points $t_q$.
The total CPU times are shown in Table \ref{tab:lphage:ttimes_large}, and the relative accuracies of the mean concentrations at the final time point w.r.t. the finer time splitting are shown in Table \ref{tab:lphage:steady_acc}.
%

\begin{figure}[h!]
\begin{minipage}[t]{0.49\linewidth}
 \centering
 \caption{Residual $||AP(t)||/||P(t)||$, and the cumulative CPU time (sec.)}
 \label{fig:lphage:Au-ttimes}
 \includegraphics[width=\linewidth]{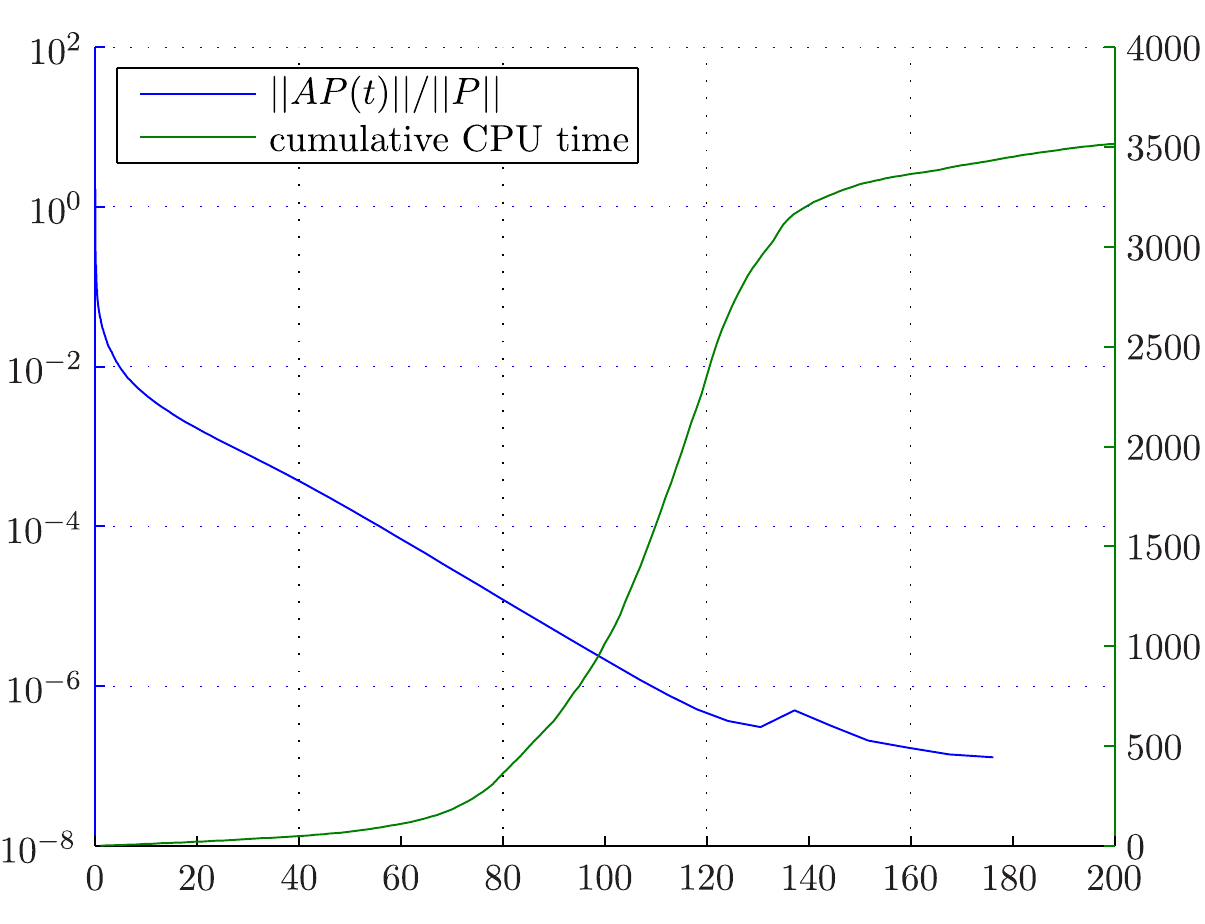}
\end{minipage}
\hfill
\begin{minipage}[t]{0.49\linewidth}
 \centering
 \caption{Average concentrations $\langle x_i\rangle$ vs. $t$ \\ \hspace*{1mm}}
 \label{fig:lphage:meanconc}
 \includegraphics[width=\linewidth]{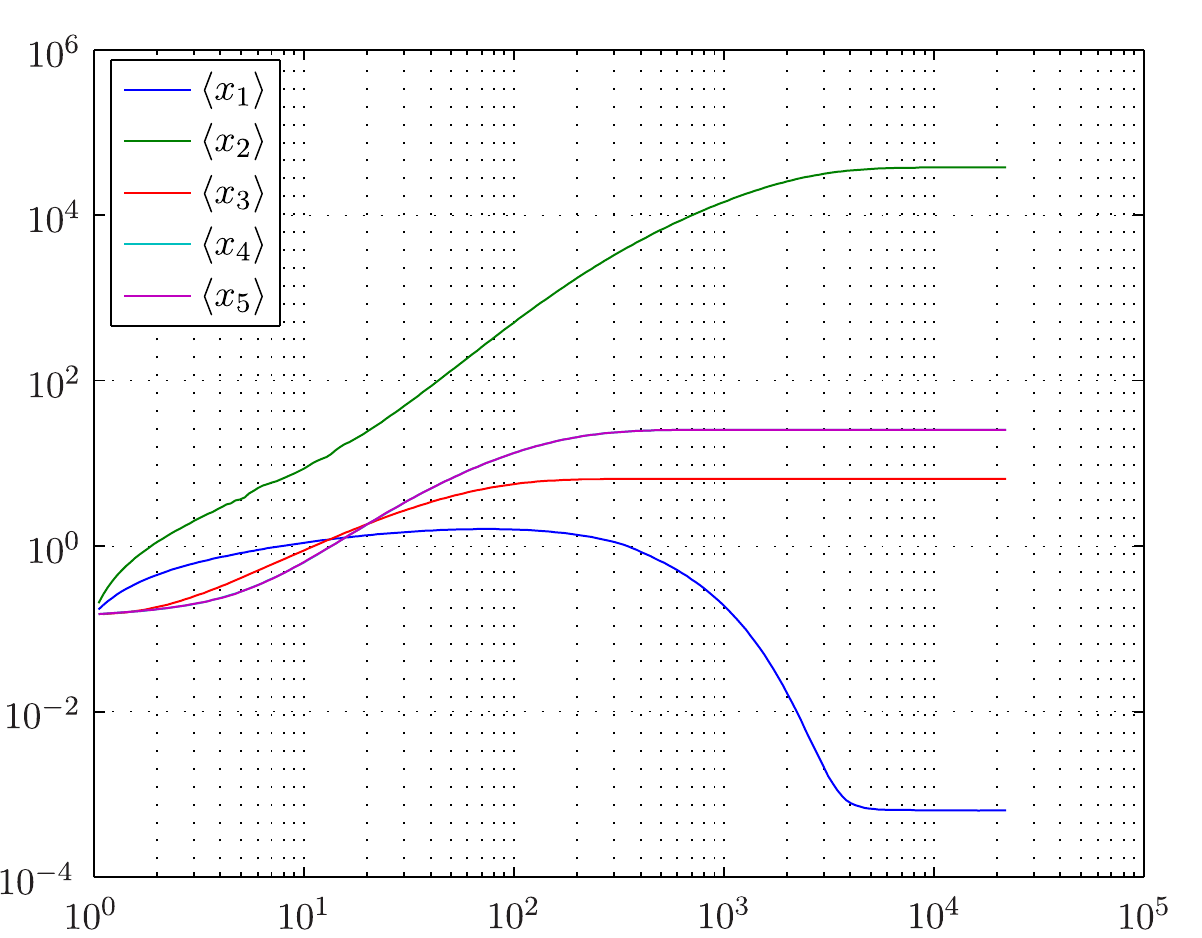}
\end{minipage}
\end{figure}

\section{Conclusions}

We have investigated the tensor product structure approach to the Chemical Master Equation.
The main contributions include:
\begin{itemize}
 \item analysis of the operators in the TT and QTT format,
 \item construction of the QTT- and QTT-Tucker-based computational algorithm for the multidimensional CME, and its complexity analysis,
 \item demonstration of the computational efficiency of the tensor-structured solution method applied to the block state-time discretized system.
\end{itemize}
The techniques have been verified on commonly used model biological systems governed by the CME, such as cascade gene networks, chemical switches, as well as more realistic ones, such as the $\lambda$-phage.
The algorithms employed for the block state-time systems \eqref{cme:eqn:global_time}, in particular, the AMEn linear solver, manifest a significant reduction of the computational time and error with respect to the previous methods, such as the time stepping iterations on tensor product manifolds, and classical Monte-Carlo-like methods (SSA).
Use of the virtual tensorisation (QTT format) allows to treat efficiently even the cases of very large state and time grids, which appear in high-concentration systems (Section \ref{sec:lphage-num}).
In addition, the simulation in presence of uncertainly (parametrically) defined coefficients has been considered.
Even with the simplest approach (coupled state-parametric system) the tensor methods have been found to be a very promising tool.
More thorough study of parametric equations, particularly in connection with the tensor product techniques, is a matter of future research.


\end{document}